\documentclass[a4paper]{amsart}
\usepackage[all]{xy}
\usepackage[colorlinks=true]{hyperref}
\usepackage{enumerate}
\usepackage{amssymb}
\newtheorem{thm}{Theorem}[section]
\newtheorem{thmz}{Theorem}             
\newtheorem{prop}[thm]{Proposition}
\newtheorem{lem}[thm]{Lemma}
\newtheorem{cor}[thm]{Corollary}

\newtheorem{conj}[thmz]{Conjecture}    

\theoremstyle{definition}
\newtheorem{defn}[thm]{Definition}
\newtheorem{ex}[thm]{Example}
\newtheorem{rem}[thm]{Remark}

\newcommand{\Rr}{\mathbb R}
\newcommand{\Ss}{\mathbb S}

\newcommand{\Ff}{\mathbb F}

\renewcommand{\d}{\mathrm d}

\newcommand{\tto}{\rightrightarrows}    
\newcommand{\rmap}{\longrightarrow}

\newcommand{\X}{\ensuremath{\mathfrak{X}}}

\newcommand{\F}{\ensuremath{\mathcal{F}}}   

\renewcommand{\O}{\ensuremath{\mathcal{O}}}

\newcommand{\G}{\mathcal{G}}            
\renewcommand{\H}{\mathcal{H}}          
\newcommand{\al}{\alpha}                
\newcommand{\be}{\beta}                 
\newcommand{\Lie}{\mathcal{L}}          
\renewcommand{\gg}{\mathfrak{g}}        

\newcommand{\lins}{\d_s}     

\DeclareMathOperator{\ad}{ad}           
\DeclareMathOperator*{\Bigoplus}{\bigoplus} 
\DeclareMathOperator{\Conn}{Conn}       
\DeclareMathOperator{\End}{End}         
\DeclareMathOperator{\Graph}{Graph}     
\DeclareMathOperator{\Hom}{Hom}         
\DeclareMathOperator{\hor}{hor}         
\DeclareMathOperator{\Hor}{Hor}         
\DeclareMathOperator{\im}{Im}           
\DeclareMathOperator{\jet}{j}           
\DeclareMathOperator{\Ker}{Ker}         
\DeclareMathOperator{\lap}{\bigtriangleup} 
\DeclareMathOperator{\modular}{mod}     
\DeclareMathOperator{\Ver}{Vert}        
\DeclareMathOperator{\vol}{vol}     

\renewcommand{\mod}{\modular}

\newcommand{\lin}{\text{lin}}           

\usepackage[dvips]{graphicx}

\newcommand{\comment}[1]{}
\begin{document}

\title{Stability of symplectic leaves}
\author{Marius Crainic}
\address{Depart. of Math., Utrecht University, 3508 TA Utrecht, 
The Netherlands}
\email{M.Crainic@math.uu.nl}

\author{Rui Loja Fernandes}
\address{Depart.~de Matem\'{a}tica, 
Instituto Superior T\'{e}cnico, 1049-001 Lisboa, Portugal} 
\email{rfern@math.ist.utl.pt}

\thanks{MC was supported in part by the NWO ``Open Competitie'' project 613.000.425. 
RLF was supported in part by FCT/POCTI/FEDER and by grants POCI/MAT/55958/2004 and POCI/MAT/57888/2004.}

\begin{abstract}
We find computable criteria for stability of symplectic leaves of Poisson manifolds. Using Poisson geometry as an inspiration, we also give a general criterion for stability of leaves of Lie algebroids, including singular ones. This not only extends but also provides a new approach (and proofs) to the classical stability results for foliations and group actions.
\end{abstract}
\date{October 2008}
\maketitle

\setcounter{tocdepth}{1}
\tableofcontents

\section*{Introduction}            %
\label{sec:introduction}           %

A Poisson structure on a manifold $M$ is a (possibly singular) foliation
by symplectic leaves of $M$. In this paper we address the question of
stability of these symplectic leaves. More precisely, given a Poisson 
structure $\pi$ on a manifold $M$, we will say that a
symplectic leaf $S$ is \emph{stable} if all nearby Poisson structures
have a nearby symplectic leaf which is diffeomorphic to $S$. In this paper we 
will give criteria for stability. Our criteria can viewed as illustrations of 
the well-known principle
\[
\emph{infinitesimal stability}\quad \Longrightarrow \quad
\emph{stability}
\] 
and are reminiscent of stability criteria in other geometric settings, such as 
foliation theory, equivariant geometry, etc. Let us recall some of these
briefly.

In \emph{foliation theory}, the classical stability results of Reeb
\cite{Rb}, Thurston \cite{Th}, Langevin and Rosenberg \cite{LaRo}, can
all be reduced to the following criterion for stability: For a fixed manifold $M$,
the set of smooth foliations of $M$ has a natural topology. One calls a leaf
$L$ of a foliation $\F$ of $M$ \emph{stable} if every nearby foliation has a
nearby leaf diffeomorphic to $L$. Also, for a base point $x\in L$, let
$\rho:\pi_1(L,x)\to GL(\nu(L)_x)$ be the linear holonomy
representation and denote by $H^1(\pi_1(L,x),\nu(L)_x)$ the
corresponding first group cohomology. Then, for a compact leaf $L$,
the condition $H^1(\pi_1(L,x),\nu(L)_x)=0$ implies that $L$ is
stable.

For \emph{group actions} one finds a similar situation. For a fixed Lie group
$G$ and manifold $M$, the set of smooth actions of $G$ on $M$ has a
natural topology. One calls an orbit $\O$ of an action $\al:G\times M\to M$
\emph{stable} if every nearby action has a nearby orbit which is
diffeomorphic to $\O$. Results of Hirsch \cite{Hr} and Stowe \cite{St}
led to the following criterion for stability: Fix $x\in \O$, let
$\rho:G_x\to GL(\nu(\O)_x)$ be the linear normal isotropy
representation and denote by $H^1(G_x,\nu(\O)_x)$ the corresponding
first group cohomology. Then, for a compact orbit $\O$, the condition
$H^1(G_x,\nu(\O)_x)=0$ implies that $\O$ is stable.

In \emph{Poisson geometry}, the problem of stability of symplectic leaves is more subtle. One 
reason is the complicated singular behavior of Poisson structures in a neighborhood of a leaf. 
Another reason is that there are actually two interesting notions of stability leading to two different questions: when is a symplectic leaf stable as a manifold (\textbf{stability}) and when is it stable as a symplectic manifold (\textbf{strong stability})? In this paper we will discuss both notions of stability. Our first result states:

\begin{thmz}[Stability]
\label{theorem1} 
Let $(M, \pi)$ be a Poisson manifold and let $S$ be a compact symplectic leaf.
If the second relative Poisson cohomology group $H^{2}_{\pi}(M, S)$ vanishes, then $S$ is stable.
\end{thmz}

Our second result deals with strong stability:

\begin{thmz}[Strong stability]
\label{theorem2} 
Let $(M, \pi)$ be a Poisson manifold and let $S$ be a compact symplectic leaf.
If the second restricted Poisson cohomology group $H^{2}_{\pi, S}(M)$ vanishes, then $S$ is 
strongly stable.
\end{thmz}

We emphasize that all the cohomology groups appearing in the statements of our theorems are computable and finite dimensional (the defining complexes are even elliptic).

It is also important to understand the relationship between stability and strong stability. 
The strong stability result (Theorem \ref{theorem2}) makes full use of the Poisson geometry setting and has no immediate consequences involving the other results we have mentioned. This may seem surprising at first sight, because one could expect the strong stability result (Theorem \ref{theorem2}) to imply the stability result (Theorem \ref{theorem1}). However, the complete statement of these results, which give the size of the space of leaves diffeomorphic to the original one, clarifies the situation: while Theorem \ref{theorem1} 
implies that nearby Poisson structures have ``many'' leaves diffeomorphic to the original one, Theorem \ref{theorem2} involves ``fewer'' leaves symplectomorphic to the original one. 

Nevertheless, we expect the two theorems to be related. This is due to the existence of a canonical map relating the relevant cohomology groups:
\[ \Phi:  H^{2}_{\pi, S}(M)\rmap H^{2}_{\pi}(M, S),\]
which led us to the following conjecture:

\begin{conj}
Let $(M,\pi)$ be a Poisson manifold and let $S$ be a compact symplectic leaf. If $\Phi$ vanishes, then any Poisson structure close enough to $\pi$ admits \emph{at least one} nearby symplectic leaf diffeomorphic to $S$. 
\end{conj}

At this point, we are not able to prove or disprove this conjecture. However, Theorem \ref{thm:necessity:Poisson} below, which gives necessary conditions for stability and strong stability, provides some support for this conjecture. A similar conjecture for group actions, regarding stability of orbits (stability) versus stability of orbit types (strong stability), was made by Stowe \cite{St}. To our knowledge this conjecture is still open.

The other aspect of our work concerns the question of finding some common background to such stability results. We will show that all these results are, in principle, just reflections of a general stability theorem for orbits of \emph{Lie algebroids}, which can be stated as follows:

\begin{thmz}[Orbit stability]
\label{theorem3} 
Let $A$ be a Lie algebroid and let $L$ be a compact orbit of $A$ with
normal bundle denoted $\nu(L)$. If $H^1(A|_{L}; \nu(L))= 0$ then $L$ is stable, i.e.,
every Lie algebroid structure which is close enough to $A$ admits a nearby orbit which is
diffeomorphic to $L$.
\end{thmz}

More complete statements of these three theorems, as well as explanations concerning the various cohomology groups appearing in the results, will be given in the body of the paper. In particular, these more complete versions will describe the size of the space of leaves diffeomorphic to the original one. 

For each specific class, Theorem \ref{theorem3} immediately gives a computable criterion for stability. However, using our approach and the specifics of the particular class being studied, one may obtain better criteria. In fact, for each of the situations considered above, the precise relationship between the various results we have mentioned, is a follows:

\subsubsection*{Foliations} In this case, Theorem \ref{theorem3} implies the classical stability theorem for foliations. This requires a passage from algebroid cohomology to group cohomology, that is standard in this case, and which will be explained in full generality in Section \ref{subsec:algbrd:grp:cohom}. Since a Lie algebroid which is close enough to (the Lie algebroid of) a foliation must be a foliation itself, the two results coincide. 

\subsubsection*{Group actions} Theorem \ref{theorem3} also implies the stability result for orbits of actions of $1$-connected (i.e., connected and simply connected) Lie groups. Without the $1$-connectedness assumption Theorem \ref{theorem3} yields a result of slightly more infinitesimal nature, which is more restrictive but also stronger than the result mentioned above: it insures the stability of orbits inside the world of infinitesimal Lie algebra actions on manifolds. One reason for this is that the Lie algebroid associated with an infinitesimal action can have arbitrary close Lie algebroids which are not associated with a Lie algebra action.

\subsubsection*{Poisson structures} Theorem \ref{theorem3} applied to Poisson geometry implies stability of leaves under the infinitesimal condition that a certain first restricted Poisson cohomology group with coefficients in the normal bundle $H^{1}_{\pi, S}(M; \nu(S))$ vanishes. This is quite different from our stability theorem in Poisson geometry (Theorem \ref{theorem1}), the reason being that one can find many Lie algebroid structures close to $T^*M$ which are not associated to a Poisson structure. Hence, unlike the case of foliations and group actions, our stability theorem in Poisson geometry \emph{is not} a direct consequence of the Lie algebroid result. In fact, Poisson geometry provides a first example of a geometric structure where the general stability theorem can be improved. Indeed, we will see that there is a canonical inclusion map:
\[ H^{2}_{\pi, S}(M) \hookrightarrow H^{1}_{\pi, S}(M;\nu(S)),\]
so Theorem \ref{theorem1} is an improvement over Theorem \ref{theorem3} (see the examples in Section \ref{sec:examp}). 
\vskip 10 pt

The proofs of the stability theorems stated above are quite different in nature from the proofs of the classical stability theorems for foliations and group actions. For these, the heart of the proof is an argument, due to Thurston, which uses Ascoli's Theorem to construct a sequence of group 1-cocycles that converge to a cocycle. In our case, the proof consists in constructing a functional, parametrized by the nearby Poisson structures, whose absolute minima are the nearby leaves diffeomorphic (or symplectomorphic) to the original one. The problem of finding extrema of this functional is solved by a geometric argument, rather than the usual method that requires constructing some minimizing sequence (that would correspond to Thurston's sequence of group cocycles) and the use of a priori estimates. An outline of our method of proof can be found in Section \ref{sec:outline} below.

This paper is organized as follows. In Section \ref{sec:results}, after introducing the necessary background material, we state the complete statments of our main results. We also show how the classical theorems for foliations and group actions follow from our stability result for Lie algebroids. In Section \ref{sec:examp}, we illustrate our results with several examples and show that our criteria are computable in many situations. In Section \ref{sec:outline}, we describe the main steps in our method of proof.  The remaining sections of the paper contain the proofs of our results.

In closing this introduction, we remark that the (much easier) case of singular points (i.e., zero dimensional leaves) was treated by us in \cite{CrFe2}, and then extended to higher order singularities by Dufour and Wade in \cite{DuWa}. 

\vskip 10pt
\noindent
\textbf{Conventions.} All our manifolds are assumed to be Hausdorff and second countable.
Our convention on the \textbf{Schouten bracket} on the space $\X^\bullet(M)=\Gamma(\wedge^\bullet TM)$ of multivector fileds on a manifold $M$ is such that, when applied to vector fields it is the usual Lie bracket, when applied to a vector field and a function it is the Lie derivative along the vector field and, furthermore, satisfies the following properties 
\begin{align*}
[X, Y]&= -(-1)^{(p-1)(q-1)}[Y, X],\\
[X, Y\wedge Z]&= [X, Y]\wedge Z + (-1)^{(p-1)q}Y\wedge [X, Z],
\end{align*}
for all $X\in \X^p(M)$, $Y\in \X^q(M)$, $Z\in \X^r(M)$. 

Another convention is that we always work over $\mathbb{R}$. For instance, all the
vector bundles will be real vector bundles, all the Hilbert spaces will be over $\mathbb{R}$ (as are the Sobolev spaces of sections of a (real) vector bundle) and $C^{\infty}(M)$ denotes the space of smooth functions $f: M\to \mathbb{R}$. Also, all our elliptic complexes are over $\mathbb{R}$. The standard properties of elliptic complexes (such as the existence of a parametrix) are usually formulated over $\mathbb{C}$ but the real version follows formally by a complexification argument (see, e.g., \cite{AS}, where a real vector space $V$ is identified with the fixed point set of the anti-linear involution $\tau: V_{\mathbb{C}}\to V_{\mathbb{C}}$ on its complexification; then real elements, real operators, etc., are recovered as the ones which are fixed by $\tau$).

\vskip 10pt
\noindent
\textbf{Acknowledgments.}
The authors would like to thank the following institutions for their hospitality and support during the 4 years that this project took shape: Centre de Recerca Matem\`atica, Erwin Schr\"odinger International Institute for Mathematical Physics, International Center for Theoretical Physics-Trieste, Universit\'e Paul Sabatier-Toulouse, Keio University-Tokyo, University of Luxembourg, University of Utrecht (R.L.F) and Instituto Superior T\'{e}cnico (M.C.).  We also thank the anonymous referees for their remarks which helped improving this paper.

\section{Main results}                  %
\label{sec:results}                     %

\subsection{Poisson structures} 
Recall (see, e.g., \cite{Wein1}) that a \textbf{Poisson structure} on a manifold $M$ can be described 
either by a Lie bracket $\{\cdot,\cdot\}$ on the space $C^\infty(M)$ satisfying the Leibniz
identity:
\[ \{f,gh\}=\{f,g\}h+g\{f,h\},\quad f,g,h\in C^\infty(M),\]
or by a bivector field $\pi\in\X^2(M)$ with
the property that
\begin{equation}
\label{eq:schouten:bracket}
[\pi,\pi]=0,
\end{equation}
where $[\cdot, \cdot]$ is the Schouten bracket (see our conventions in the introduction). 
These two ways of looking at Poisson structures are related by: $\pi(\d f, \d g)=\{f, g\}$. 
We will often interpret $\pi$ as a vector bundle map $\pi^\#:T^*M\to TM$.

The description of Poisson structures in terms of bivectors allows us to talk about Poisson structures which are $C^{\kappa}$-close ($1\le \kappa\le+\infty$) to a given one: this refers to the \textbf{$C^{\kappa}$ compact-open topology} (also known as the \emph{$C^{\kappa}$ weak topology}) in $\X^2(M)$, i.e., the topology of uniform convergence over compact subsets of maps and its derivatives up to order $\kappa$ (see, e.g., \cite{Hr2} for the precise definition).

In a Poisson manifold $(M, \pi)$ any function $f\in C^{\infty}(M)$ induces a vector field
\[ X_{f}:= \{f, \cdot \},\]
called the \textbf{Hamiltonian vector field} induced by $f$. Moving along Hamiltonian flows gives a partition of  $M$ into the so called \textbf{symplectic leaves}: two points of $M$ are in the same leaf if and only if they can be reached by following the flow of (one or more) Hamiltonian vector fields. The leaves carry a canonical smooth structure which make them into regular immersed submanifolds of $M$. Their name is due to the fact that they are the maximal connected integral leaves of the singular integrable distribution on $M$ defined by the image of $\pi^{\sharp}$. It follows that each such leaf $S$ has a canonical symplectic form
\[ \omega_S(X_f,X_g)=\{f,g\} \quad \text{(on $S$)}.\] 
Hence, roughly speaking, a Poisson structure is a ``singular symplectic foliation''. However,
there is interesting geometry also in the transverse direction. For any symplectic leaf $S$, each fiber of the normal bundle
\[ \nu_{S}:= T_SM/TS\]
carries a canonical linear Poisson structure (the linearization of $\pi$ in the transverse direction). Equivalently, the conormal bundle $\nu_{S}^{*}$ is a Lie algebra bundle: at each $x\in M$, the bracket can be described by
\[ [\d_xf, \d_xg]= \d_x\{f, g\},\]
for $f,g$ locally defined functions around $x$ and constant on $S$. 

The symplectic form on the leaf $S$ and the transverse linear Poisson structures on $\nu(S)$ can be assembled together into a single Poisson structure $\jet^1_S\pi$ on a tubular neighborhood of $S$, which we will call the \textbf{first jet approximation} to $\pi$ along $S$. We will describe later in Section \ref{sec:linear:approx} how $\jet^1_S\pi$ is constructed. 

For a Poisson manifold $(M, \pi)$, one defines the \textbf{Poisson cohomology} $H^{\bullet}_{\pi}(M)$ as the cohomology of the complex of multivector fields $(\X^{\bullet}(M), \d_{\pi})$, where:
\[ \d_{\pi}:=[\pi, \cdot].\]
Fixing a symplectic leaf $S$, one obtains a complex $(\X^{\bullet}_{S}(M),\d_{\pi}|_{S})$ 
where $\X^{\bullet}_{S}(M)= \Gamma (\Lambda^{\bullet} TM|_{S})$ is the space of 
multivector fields along $S$ (it is not difficult to see that $\d_{\pi}$ does restrict to $S$). The corresponding cohomology is called the \textbf{Poisson cohomology restricted to $S$}(\footnote{This cohomology was introduced in \cite{GiLu}, where it was called the \emph{relative Poisson cohomology}. We reserve this name for yet another cohomology we will be using.}) and will be denoted  by $H^{\bullet}_{\pi, S}(M)$. There is an obvious restriction map:
\[ H^\bullet_\pi(M)\rmap H^\bullet_{\pi, S}(M).\]

Note that $\X^{\bullet}_{S}(M)$ contains the de Rham complex of $S$ as a subcomplex, with the inclusion induced by dualizing the anchor map $\pi^{\sharp}: T^{*}_{S}M\to TS$. We set:
\[ \X^{\bullet}(M, S):=\X^\bullet_S(M)/\Omega^\bullet(S),\]
and the resulting cohomology is called the \textbf{Poisson cohomology relative to $S$}, denoted $H^{\bullet}_{\pi}(M, S)$. The map induced by the projection will be denoted 
\[ \Phi: H^{\bullet}_{\pi, S}(M)\rmap H^{\bullet}_{\pi}(M, S).\]
These are the cohomologies appearing in the first two theorems in the introduction, for which we can now state the complete versions:

\begin{thm}
\label{theorem11}   
Let $S$ be a $n$-dimensional compact symplectic leaf of a Poisson manifold $(M, \pi)$ which satisfies
$H^{2}_{\pi}(M, S)= 0$, and let $\kappa>\frac{n}{2}+1$ be an integer. For any neighborhood $V$ of
$S$ there exists a neighborhood $\mathcal{V}$ of $\pi$ in the $C^\kappa$-topology such that
any Poisson structure $\theta$ in $\mathcal{V}$ has a family of symplectic leaves in $V$, diffeomorphic to $S$, smoothly parametrized by $H^1_\pi(M,S)$
and depending continuously on $\theta$. 
\end{thm}

The parameter space has a nice geometric interpretation: $H^1_\pi(M,S)$ coincides with the space of leaves of the first jet approximation $j^1_S\pi$ which project diffeomorphically to $S$. Moreover, we have $H^1_\pi(M,S)=\Gamma_\text{flat}(\nu^0(S))$, the space of flat sections of the subbundle of $\nu(S)$ where the linear transverse Poisson structure vanishes (see Corollary \ref{cor:leaves} (i)).

\begin{thm}
\label{theorem22}   
Let $S$ be a $n$-dimensional compact symplectic leaf of a Poisson manifold $(M, \pi)$ which satisfies
$H^{2}_{\pi, S}(M)= 0$, and let $\kappa>\frac{n}{2}+1$ be an integer. For any neighborhood $V$ of
$S$ there exists a neighborhood $\mathcal{V}$ of $\pi$ in the $C^\kappa$-topology such that
any Poisson structure $\theta$ in $\mathcal{V}$ has a family of symplectic leaves in $V$, symplectomorphic to $S$, smoothly parametrized by the image of $\Phi:H^1_{\pi,S}(M)\to H^1_\pi(M,S)$ and depending continuously on $\theta$. 
\end{thm}

The parameter space of strongly stable leaves (Theorem \ref{theorem22}) is a subspace of the parameter space for stable leaves (Theorem \ref{theorem11}), as it should be. It also has a nice geometric interpretation: it coincides with the space of leaves of the first jet approximation $j^1_S\pi$ which project diffeomorphically to $S$ and which have symplectic form isotopic to $\omega_S$ (see Corollary \ref{cor:leaves} (ii))

One obvious question that one could ask in the Poisson setting is wether the conditions in Theorems \ref{theorem11} and \ref{theorem22} are in any sense necessary for stability to hold. It turns out that there is a natural class of Poisson structures, namely those behaving linearly in a neighborhood of the leaf $S$, where these conditions appear naturally. More precisely, let $\pi$ be a Poisson structure on a vector bundle $p:E\to S$, such that the zero section is a leaf of $\pi$. We will say that $\pi$ is of \textbf{first order around $S$} if $\pi$ is Poisson diffeomorphic to $j^1_S\pi$ in some neighborhood of $S$. Then we have:

\begin{thm}
\label{thm:necessity:Poisson}
Let $\pi$ be a Poisson structure which is of first order around a compact symplectic leaf $S$. Then:
\begin{enumerate}[(i)]
\item If $S$ is stable, then the map $\Phi:H^2_{\pi,S}(M)\to H^2_\pi(M,S)$ vanishes.
\item If $S$ is strongly stable, then $H^2_{\pi,S}(M)=0$.
\end{enumerate} 
\end{thm}

Notice that this result provides some evidence to the conjecture stated in the Introduction.

\subsection{Lie algebroids} We now turn to the discussion of the complete version of Theorem \ref{theorem3}. 

Recall first a \textbf{Lie algebroid} over a manifold $M$ consists of a vector bundle $A$ over $M$ together with a Lie algebra bracket $[\cdot , \cdot]$ on the space $\Gamma(A)$ of sections of $A$ and a bundle map $\rho: A\to TM$ (called the \textbf{anchor}) satisfying the Leibniz identity
\[ 
[\alpha, f\beta]= f[\alpha, \beta]+ \Lie_{\rho(\alpha)}(f)\beta,\quad (\alpha,\beta\in\Gamma(A), f\in C^{\infty}(M)).
\]
Here, we denote by $\Lie_X$ the Lie derivative along the vector field $X\in \X(M)$. 

\begin{ex} The main examples for this paper are:
\begin{description}
\item[{\rm \emph{Foliations}}] Any foliation on $M$ can be viewed as a Lie algebroid with injective anchor map: $A$ is the bundle of vectors tangent to the leaves and the anchor is the inclusion. 
\item[{\rm \emph{Infinitesimal actions}}] To any infinitesimal action $\rho: \mathfrak{g}\to \X(M)$ of a finite dimensional Lie algebra $\mathfrak{g}$ on a manifold $M$ one associates a Lie algebroid: $A$ is the trivial bundle over $M$ with fiber $\mathfrak{g}$, the anchor is $\rho$ and the bracket is uniquely determined by the condition that it restricts on constant sections to the Lie bracket of $\mathfrak{g}$. 
\item[{\rm \emph{Poisson manifolds}}] Associated to any Poisson structure $\pi$ on $M$, there is a Lie algebroid structure on the cotangent bundle $T^*M$: the anchor is just $\pi^{\sharp}:T^*M\to TM$, while the bracket $[\cdot, \cdot]_{\pi}$ on $\Gamma(T^*M)= \Omega^1(M)$ is uniquely determined by
\[ [\d f, \d g]_{\pi}= \d\{f, g\}.\]
\end{description}
\end{ex}

Generalizing the leaves of foliated manifolds, orbits of actions and the symplectic leaves of Poisson manifolds, one can talk about the \textbf{leaves of a Lie algebroid} $A$ over $M$: they are regular immersed submanifolds $L$ of $M$ which are the maximal integrals of the singular distribution $\rho(A)\subset TM$:
\[ T_xL= \rho(A_x) \quad \forall \ x\in L .\]
Two points $x,y\in M$ are in the same leaf if and only if they can be joined by an \textbf{$A$-path}, i.e., there exists path $a:I\to A$, covering a $C^1$ base path 
$\gamma:I\to M$, starting at $x$ and ending at $y$ (meaning $\gamma(0)=x$ and $\gamma(1)=y$), such that:
\[ \frac{\d \gamma}{\d t}(t)=\rho(a(t)), \quad \forall t\in I.\]
Given a leaf $L\subset M$ of $A$, one can restrict $A$ to $L$ to obtain a new algebroid $A|_{L}$ over $L$, with the bracket uniquely determined by: 
\[ [\alpha|_{L}, \beta|_{L}]= [\alpha, \beta]|_{L}.\]
This restricted algebroid $A|_L$ has the important property that it is transitive (i.e., its
anchor is surjective). Also, the bracket on $\Gamma(A)$ restricts to a pointwise
bracket on the kernel of the anchor map at each point $x\in M$, to induce a Lie algebra
\[ \mathfrak{g}_x(A):= \textrm{Ker}(\rho_x)\subset A_x,\]
called the isotropy Lie algebra at $x$. When $x$ runs inside a leaf $L$, they fit
together into a smooth Lie algebra bundle $\mathfrak{g}_{L}(A)$ over $L$, which can be viewed as a Lie algebroid with zero anchor. One obtains a short exact sequence of Lie algebroids:
\[ \xymatrix{0\ar[r] & \mathfrak{g}_{L}(A)\ar[r] & A|_L\ar[r] & TL\ar[r] & 0}.\]
If we fix a (vector bundle) splitting of the sequence, $\sigma:TL\to A|_L$, we can define a connection $\nabla^\sigma$ on $\mathfrak{g}_{L}(A)$ by:
\[ \nabla^\sigma_X\xi=[\sigma(X),\xi], \quad (X\in\X(S),\xi\in\Gamma(\mathfrak{g}_{L}(A))).\]
The connections induced on the center $Z(\mathfrak{g}_{L})(A)$ and also on the abelianization $\mathfrak{g}_{L}(A)/[\mathfrak{g}_{L}(A),\mathfrak{g}_{L}(A)]$ are flat and do not depend on the choice of splitting. Hence, these two bundles are canonically flat bundles over $L$.

\begin{ex}
\label{ex:short:seq:Poisson}
When $A=T^*M$ is the Lie algebroid associated to a Poisson manifold $(M,\pi)$ and we fix a symplectic leaf $S$, the sequence above amounts to:
\[ \xymatrix{0\ar[r] & \nu^*(S)\ar[r] & T_S^*M\ar[r] & TS\ar[r] & 0} \]
We conclude that the bundle $\nu_S^0$ consisting of zeros of the transverse linear Poisson structure is canonically a flat vector bundle over $S$. The sections of this bundle form the parameter space in Theorem \ref{theorem11}.
\end{ex}

For any Lie algebroid $A$ one has an associated de Rham complex: $\Omega^\bullet(A)= \Gamma(\wedge^\bullet A)$ with the differential $\d_A$ given by the classical Koszul formula
\begin{align*} 
\d_A\omega(\alpha_1, \ldots , \alpha_{q+1})=&\sum_{i}(-1)^{i+1} \Lie_{\rho(\alpha_i)}(\omega(\alpha_1, \ldots , \widehat{\alpha_i}, \ldots , \alpha_{q+1}))+\\
& + \sum_{i< j} (-1)^{i+j}\omega([\alpha_i, \alpha_j], \ldots , \widehat{\alpha_i}, \ldots, \widehat{\alpha_j}, \ldots , \alpha_{q+1}).
\end{align*}
The resulting cohomology $H^\bullet(A)$ is called the \textbf{$A$-de Rham cohomology}. 

\begin{ex} The $A$-de Rham cohomology generalizes the usual de Rham cohomology (when $A= TM$),
foliated cohomology (when $A$ comes from a foliation), Lie algebra cohomology (when $M$ is
a point and $A$ is a Lie algebra). When $(M,\pi)$ is a Poisson manifold and
$A=T^*M$ is the associated Lie algebroid, then $H^{\bullet}(A)$ coincides with the Poisson cohomology $H^{\bullet}_{\pi}(M)$. For any symplectic leaf $S$, the $A$-de Rham cohomology $H^{\bullet}(A|_S)$ of the restricted Lie algebroid $A|_S=T_S^* M$ coincides with $H^{\bullet}_{\pi,S}(M)$, the Poisson cohomology restricted to $S$. 
\end{ex}

There is a simple variation of $A$-de Rham cohomology obtained by allowing coefficients. Here by coefficients we mean representations of $A$, i.e., vector bundles $E$ over $M$ endowed with $A$-derivatives operators:
\[ \Gamma(A)\otimes \Gamma(E)\rmap \Gamma(E),\ \ (\alpha, e)\mapsto \nabla_{\alpha}(e), \]
which satisfy the connection-like identities
\[  \nabla_{f\alpha}(e)= f\nabla_{\alpha}(e),\quad \nabla_{\alpha}(fe)= f\nabla_{\alpha}(e)+ \Lie_{\rho(\alpha)}(f) e,\quad (f\in C^\infty(M))\]
as well as the flatness condition
\[ \nabla_{[\alpha, \beta]}= \nabla_{\alpha}\nabla_{\beta}- \nabla_{\beta}\nabla_{\alpha}.\]
One defines the complex $\Omega^\bullet(A;E)=\Gamma(\wedge^\bullet A^*\otimes E)$, with the differential $\d_A$ given by a formula similar to the one above, but where we 
replace $\Lie_{\rho(\alpha_i)}$ (which does not longer make sense) by $\nabla_{\alpha_i}$.

\begin{ex}
If we fix a leaf $L\subset M$ of $A$, the restricted algebroid $A|_L$ has a canonical representation on the normal bundle $\nu_{L}= T_{L}M/TL$:
\[ \nabla_{\alpha}(X| \mod TL)= [\rho(\alpha), X]  \mod TL.\]
This representation is known as the \textbf{Bott representation} of $A|_L$ and
the resulting cohomology $H^*(A|_{L};\nu_{L})$ is the one relevant to the stability 
of leaves of Lie algebroids. Note that in degree zero one has
\[ H^0(A|_{L};\nu_{L})= \Gamma_{\text{flat}}(\nu_L),\]
the space of flat sections (with respect to the Bott connection).
\end{ex}

There is one final ingredient in the statement of our stability theorem for Lie algebroids that we need to introduce, namely the topology on the set of Lie algebroid structures on a fixed vector bundle. First, note that locally (over $M$) we can describe a Lie algebroid $A\to M$ by certain \textbf{structure functions} as follows: fix local coordinates $(U,x^1,\dots,x^n)$ on $M$ and a basis of sections $\{e_1,\dots,e_r\}$ of $A|_U$. Then, 
there are functions $a^\al_i,c_{ij}^k\in C^\infty(U)$ such that:
\[ \rho(e_i)=a_i^\al\frac{\partial}{\partial x^\al},\quad [e_i,e_j]=c_{ij}^k e_k.\]
This makes it possible to compare any two Lie algebroid structures over $U$ by comparing the corresponding structure functions (in whatever $C^{\kappa}$-topology one may wish). Now, if one wants to do this globally over $M$, one can choose a connection $\nabla$ on the vector bundle and proceed as follows: to each Lie algebroid structure, with anchor $\rho$ and Lie bracket $[~,~]$, we associate sections $a\in \Gamma(A^*\otimes TM)$ and 
$c\in\Gamma(\wedge^2A^*\otimes A)$, by setting:
\[ 
\langle a,(\alpha,\omega)\rangle=\langle\rho(\al),\omega\rangle,\quad 
\langle c,(\alpha,\beta,\xi)\rangle=\langle[\al,\be]-\nabla_{\rho(\al)}\be,\xi\rangle.
\]
The $C^{\kappa}$-topology on sections of a vector bundle now induce a \textbf{$C^\kappa$-topology on the space of Lie algebroid structures} on $A$.

\begin{thm}
\label{theorem33} 
Let $L$ be a $n$-dimensional compact leaf a Lie algebroid $A$ which satisfies $H^1(A|_{L};\nu_{L})= 0$, and let $\kappa>\frac{n}{2}$ be an integer. For any neighborhood $V$ of $L$ there exists a neighborhood $\mathcal{V}$ of the Lie algebroid $A$ in the $C^\kappa$-topology such that any Lie algebroid structure in $\mathcal{V}$ has a family of leaves in $V$, diffeomorphic to $L$, smoothly parametrized by $H^0(A|_{L};\nu_{L})$ and depending continuously on the algebroid structure. 
\end{thm}

Similar to the Poisson case, the parameter space can also be characterized as the space of leaves of the \emph{first jet approximation} to the Lie algebroid $A$ along $L$. Moreover, for Lie algebroid structures of \emph{first order type} around $L$ the condition in the theorem is also a necessary condition for stability.

\subsection{From Lie algebr(oids) to Lie group(oids)}
\label{subsec:algbrd:grp:cohom}
Let us explain how the Lie algebroid cohomology group appearing in the theorem can be re-interpreted as differentiable cohomology of a certain Lie group. This will show, for example, that the classical stability theorem for foliations follows from our general theorem for Lie algebroids.

The relevant group is the \textbf{first homotopy group of $A$} based at $x$ (or the \textbf{fundamental group of $A$} based at $x$) which will be denoted $G_x(A)$. The elements of $G_x(A)$ are the equivalence classes of $A$-paths which start and end at $x$, i.e., the $A$-loops based $x$, modulo the equivalence relation of $A$-homotopy. The space $P(A,x)$ formed by all $A$-loops based at $x$ is an infinite dimensional Banach manifold and $G_x(A)$, endowed with the quotient topology, becomes a topological group ((see Sections 1.1 and 1.3 in \cite{CrFe1}). By the smooth structure on $G_x(A)$ we will allways mean a smooth structure which makes the quotient map $P(A,x)\to G_x(A)$ into a submersion. Such a smooth structure, if it exists, is unique and makes $G_x(A)$ into a Lie group integrating $\mathfrak{g}_x(A)$. In general, $G_x(A)$ is neither connected nor simply-connected.

Geometrically, $A$-paths can be used to perform parallel transport with respect to $A$-connections. The resulting maps are invariant under $A$-homotopy provided the $A$-connection is flat (i.e., if it defines a representation of $A$) \cite[Proposition 1.6]{CrFe1}. This applied to the Bott representation of $A|_L$ gives a representation of the fundamental group of $A$:
\[ \textrm{hol}: G_x(A)\rmap GL(\nu_x),\]
called the \textbf{linear holonomy representation} based at $x$.



\begin{prop}
\label{prop-isotropy}
Let $A$ be a Lie algebroid, $L$ a leaf of $A$ and $x\in L$. If $G_x(A)$ is smooth then 
\[ H^1(A;\nu_L)\cong H^1(G_x(A); \nu_x)\]
where the right hand side is the differentiable cohomology of the Lie group $G_x(A)$ with coefficients in the
linear holonomy representation.
\end{prop}

\begin{proof}
First of all, the condition that $G_x(A)$ is smooth is equivalent to the integrability
of the Lie algebroid $A_L$ (\cite{CrFe1}). Let $\G_L=\G(A_L)$ be the fundamental groupoid of $A_L$, formed by $A$-homotopy classes of arbitrary $A$-paths (also known as the Weinstein groupoid of $A$). It is the unique (up to isomorphism) $s$-simply connected Lie groupoid
integrating $A_L$. 

Similarly to the discussion above, parallel transport with respect to the Bott connection makes $\nu_L$ into a representation of the groupoid $\G_L$. Since the $s$-fibers of $\G_L$ are 1-connected, a theorem of \cite{Cr} states that the Van Est map
\[ \textrm{VE}: H^1(\G_L;\nu_L)\rmap H^1(A_L;\nu_L),\]
is an isomorphism. Here, as in loc. cit., $H^{\bullet}(\G_L;\nu_L)$ is defined via the complex of smooth (groupoid) cocycles on $\G_L$ with coefficients in
$\nu_L$ (differentiable cohomology). 

Next, by its very construction, $G_x(A)$ coincides with the isotropy group of $\G_L$ (i.e., the Lie group of arrows of $\G_L$ starting and ending at $x$). Hence $\G_L$, being a transitive Lie groupoid, is Morita equivalent to $G_x(A)$. Under this equivalence, the representation of $\G_L$ on $\nu_L$ corresponds to the linear holonomy representation of $G_x(A)$ on $\nu_x$.

Finally, since the first differentiable groupoid cohomology is a Morita invariant (see, e.g., \cite{Cr}), we conclude that $H^1(\G_L; \nu_L)$ and $H^1(G_x(A);\nu_x)$ are isomorphic and the result follows.
\end{proof}

We can now look back at foliations, groups actions and Poisson manifolds to see what we can conclude from Theorem \ref{theorem33}.

\subsubsection{Foliations}
Combining Proposition \ref{prop-isotropy} with Theorem \ref{theorem33} we recover the classical stability theorem for foliations, as we now describe. 

Let $A=T\mathcal{F}$ be a foliation, viewed as an involutive subbundle of $TM$ and as an algebroid with the inclusion as anchor map. Then the Bott representation becomes the usual flat connection on $\nu_L$,  and we have
\[ H^1(A;\nu_L)= H^1(L;\nu_L) \text{ and } G_x(A)= \pi_1(L,x).\] 
Also, the topology on the space of Lie algebroid structures introduced above, induces the usual $C^{\kappa}$-topology on the space of foliations. Hence:

\begin{cor}
Let $L$ be a compact leaf of a foliation $\F$, let $\kappa>\frac{n}{2}+1$ be an integer, and assume that 
\[ H^1(\pi_1(L,x);\nu_x)=0.\] 
Then for any neighborhood $V$ of $L$ there exists a neighborhood $\mathcal{V}$ of the foliation $\F$ in the $C^{\kappa}$-topology such that any foliation in $\mathcal{V}$ has a family of leaves in $V$, diffeomorphic to $L$, smoothly parametrized by \emph{$\Gamma_{\text{flat}}(\nu_{L})$}.
\end{cor}

\subsubsection{Group actions}
We also obtain a stability result for actions of $1$-connected Lie groups. In fact, when 
$A= \mathfrak{g}\ltimes M$ is the Lie algebroid associated to a $1$-connected Lie group $G$ acting on $M$ then $G_x(A)\cong G_x$, so that for any orbit $\O$ of $G$, we have $H^1(A;\nu_\O)$ and $H^1(G_x;\nu_x)$ are isomorphic (here $\nu_x$ is the normal space to the orbit $\O$ through $x$). The set of smooth actions of $G$ also carries a natural $C^{\kappa}$-topology, and we conclude:

\begin{cor}
Let $\O$ be a compact orbit of an action of a $1$-connected Lie group $G$ on $M$, let $\kappa>\frac{n}{2}+1$ be an integer, and assume that 
\[ H^1(G_x;\nu_x)=0.\] 
Then for any neighborhood $V$ of $\O$ there exists a neighborhood $\mathcal{V}$ of the action in the $C^{\kappa}$-topology such that any action in $\mathcal{V}$ has a family of leaves in $V$, diffeomorphic to $\O$, smoothly parametrized by \emph{$\Gamma_{\text{flat}}(\nu_{\O})$}.
\end{cor} 

Notice that a Lie algebroid structure which is close to a Lie algebroid $T\F$ of a foliation must have injective anchor, and hence it is the Lie algebroid of a foliation. So the notions of stability for a foliation $\F$ in the sense of foliations and in the sense of Lie algebroids coincide. 

This is not anymore the case for Lie algebra actions: a Lie algebroid which is close to an action Lie algebroid $\gg\ltimes M$ need not be an action Lie algebroid. However, the first order approximation to a Lie algebroid along an orbit is an action algebroid. This explains why we still recover (infinitesimally) the classical stability result for actions.

\subsubsection{Poisson structures}
Finally, in the case of Poisson structures the two stability results do differ. There is a precise relationship between Theorem \ref{theorem11} and Theorem \ref{theorem33}, which we now explain. Recall that $H^{\bullet}_{\pi, S}(M;\nu_S)$ denotes the cohomology $H^{*}(T^*_SM;\nu_S)$, and it is computed by the complex
\[ \X^{\bullet}_{S}(M;\nu_S)= \Gamma(\wedge^{\bullet}T_{S}M\otimes \nu_S),\]
where the differential is induced by the operation of bracketing with $\pi$. There is a canonical inclusion
\begin{align*} 
i: \X^{\bullet}(M, S)&\rmap \X^{\bullet-1}_{S}(M; \nu_S)\\
X_1\wedge \ldots \wedge& X_k\mapsto \sum_{i} (-1)^{i+1} X_1\wedge \ldots \widehat{X_i} \ldots \wedge X_k \otimes (X_i \mod T_S)
\end{align*}
and this map is compatible with the differentials. This leads to:

\begin{lem} 
The inclusion $i$ induces an injection $H^{2}_{\pi}(M, S)\hookrightarrow H^{1}_{\pi, S}(M;\nu_S)$.
\end{lem}

This inclusion can be seen as a reflection of the fact that a Lie algebroid structure on the cotangent bundle $T^*M$ is associated with a Poisson structure iff the anchor is skew-symmetric and the bracket of closed 1-forms is also a closed 1-form. There are Poisson structures $\pi$ with a symplectic leaf $S$ for which every nearby Poisson structure has nearby symplectic leaves diffeomorphic to $S$, while there exist some nearby Lie algebroids with no leaves diffeomorphic to $S$ (see Example \ref{ex:not:algbd:stable} below).
  
The conclusion is that the criteria for stability of symplectic leaves in Theorem \ref{theorem11} is stronger than the criteria arising from the general algebroid setting given by Theorem \ref{theorem33}.

\begin{rem}
One should also try to apply our results to recover at least some of the Kodaira-Spencer stability theory of complex structures. In \cite{Ko}, for example, Kodaira studies deformations of a complex submanifold $W$ of a complex manifold $V$. By looking at the holomorphic Lie algebroid over $V$ whose sections are the vector fields tangent to the submanifold $W$ (which has $W$ as leaf) one should be able to obtain a stability theorem for such submanifolds. This requires developing the analogue of our results in the complex holomorphic setting, so we leave this for future work.
\end{rem}

\section{Examples}     
\label{sec:examp}

We will now illustrate our main theorems with a few examples. Concerning terminology,
we will say that a symplectic leaf $S$ of a Poisson manifold $(M, \pi)$ is:
\begin{itemize}
\item \textbf{$C^{\kappa}$-stable} if for any neighborhood $V$ of $S$, there exists a neighborhood  $\mathcal{V}$ of $\pi$ in the $C^{\kappa}$-topology such that any $\theta\in \mathcal{V}$ has a leaf \emph{diffeomorphic} to $S$ and contained in $V$. 
\item \textbf{strongly $C^{\kappa}$-stable} if the same condition holds but with    
 ``\emph{diffeomorphic}'' replaced by ``\emph{symplectomorphic}''. 
\item \textbf{algebroid $C^{\kappa}$-stable} if for any neighborhood $V$ of $S$, there exists a neighborhood $\mathcal{V}$ of $(T^*M,[~,~]_\pi,\pi^\sharp)$ in the $C^{\kappa}$-topology such that any Lie algebroid structure in $\mathcal{V}$ has a leaf \emph{diffeomorphic} to $S$ and contained in $V$.
\end{itemize}
When $\kappa=\infty$, we drop $\kappa$ from the notation.
Our theorems imply the following criteria for stability: for a compact symplectic leaf $S$ and $\kappa>\frac{n}{2}+1$, one has:
\begin{description}
\item[Criterion 1] If $H^{2}_{\pi}(M,S)=0$ then $S$ is $C^{\kappa}$-stable.
\item[Criterion 2] If $H^{2}_{\pi,S}(M)=0$ then $S$ is strongly $C^{\kappa}$-stable.
\item[Criterion 3] If $H^{1}_{\pi,S}(M,\nu_{S})=0$ then $S$ is algebroid $C^{\kappa-1}$-stable.
\end{description}

Before we turn to the examples, let us also point out that for a Poisson manifold $(M, \pi)$ with $x\in M$ a singular point of $\pi$ (i.e, a 0-dimensional leaf $S=\{x\}$),
the relevant cohomologies are related to the isotropy Lie algebra $\mathfrak{g}_x$
of $\pi$ at $x$ as follows:
\begin{align*} 
&H^{2}_{\pi}(M,x)=H^{2}_{\pi, x}(M)\cong H^{2}(\mathfrak{g}_x), \\ 
&H^{1}_{\pi, x}(M,\nu_{x})=H^{1}(\mathfrak{g}_x;\mathfrak{g}_x).
\end{align*}
In particular, let $M=\mathfrak{g}^*$ be the dual of a Lie algebra endowed with 
its linear Poisson structure. The origin $x=0$ is a singular point and we have that $\mathfrak{g}_x= \mathfrak{g}$, so:
\[ 
H^{2}_{\pi, x}(M)=H^2(\gg),\quad H^{1}_{\pi, x}(M,\nu_{x})=H^{1}(\mathfrak{g};\mathfrak{g}).
\]
This will be used in the examples to follow.

\begin{ex}
On $\mathbb{R}^2$, the Poisson structures
\[ \pi_{\epsilon}= (x^2+ y^2+ \epsilon) \frac{\partial}{\partial x}\wedge \frac{\partial}{\partial y}\]
are symplectic when $\epsilon > 0$, while $\pi_{0}$ has $S=\{0\}$ as a singular point. Hence the origin is unstable for $\pi_{0}$. The cohomologies $H^{2}_{\pi_0}(M,S)$, $H^{2}_{\pi_0,S}(M)$ and $H^{1}_{\pi_0,S}(M,\nu_{S})$ are all non-zero (note that, in this example, $\gg_0$ is the 2-dimensional abelian Lie algebra). 
\end{ex}

\begin{ex}
\label{ex:not:algbd:stable}
On $\mathbb{R}^2$, the linear Poisson structure
\[ \pi= x \frac{\partial}{\partial x}\wedge \frac{\partial}{\partial y}\]
has the origin stable, but not algebroid stable.

The stability follows by a direct argument (or by applying Criterion 1). 
To show algebroid unstability, note that the Lie algebroid associated to $\pi$ can be described in terms of the coframe $e_1=\d x$, $e_2=\d y$, by:
\begin{align*}
&[e_1, e_2]=e_1,\\
\rho(e_1)&= x\frac{\partial}{\partial y},\quad 
\rho(e_2)=-x \frac{\partial}{\partial x}.
\end{align*}
One can find a Lie algebroid $(T^*\Rr^2,[\cdot, \cdot]_{\epsilon}, \rho_{\epsilon})$ arbitrary close to this one which has the same Lie bracket and a nonvanishing anchor, given by:
\[ 
\rho_{\epsilon}(e_1)= x\frac{\partial}{\partial y},\quad 
\rho_{\epsilon}(e_2)= - x \frac{\partial}{\partial x}+ \epsilon \frac{\partial}{\partial y}.
\]
Since $M=\gg^*$, where $\gg$ is the non-abelian 2-dimensional Lie algebra, in this example: 
\[ H^{2}_{\pi,x}(M)=0,\quad H^{1}_{\pi,x}(M,\nu_x)\neq 0. \]
\end{ex} 
 
\begin{ex}
Let $M=\mathfrak{su}(2)^*=\mathbb{R}^3$ with its linear Poisson structure
\[ \pi= x\frac{\partial}{\partial y}\wedge \frac{\partial}{\partial z}+ y \frac{\partial}{\partial z}\wedge \frac{\partial}{\partial x}+ z\frac{\partial}{\partial x}\wedge \frac{\partial}{\partial y}.\]
Its symplectic leaves are the origin and the spheres centered at the origin. All the leaves are both strongly stable and algebroid stable. The same happens with the linear Poisson structure on $\gg^*$, for any Lie algebra $\mathfrak{g}$ which is compact and semi-simple. This should be compared with Conn's linearization theorem \cite{Conn}.

In fact, let $M=\mathfrak{g}^*$ be the dual of a Lie algebra $\mathfrak{g}$ and let $G$ be the $1$-connected Lie group integrating $\gg$. The symplectic leaf through a point $\xi\in \mathfrak{g}$ is precisely the coadjoint orbit $\mathcal{O}_{\xi}$ containing $\xi$. Also, the cotangent Lie algebroid $T^*M$ coincides with the action Lie algebroid associated to the coadjoint action, so that $T^*M$ integrates to the action Lie groupoid $\G=G\ltimes\gg^*\tto\gg^*$. If $\gg$ is compact and semi-simple, so that $G$ is compact, the Van Est map gives isomorphisms (see the proof of Proposition \ref{prop-isotropy}):
\begin{align*}
&H^{2}_{\pi,\mathcal{O}_{\xi}}(M)\cong H^2(\G|_{\mathcal{O}_{\xi}}),\\
&H^1_{\pi,\mathcal{O}_{\xi}}(M;\nu(\mathcal{O}_{\xi}))\cong H^1(\G|_{\mathcal{O}_{\xi}};\nu(\mathcal{O}_{\xi})
\end{align*}
and both groups vanish (since $\G|_{\mathcal{O}_{\xi}}$ is proper).
\end{ex}

\begin{ex}
Let $(S,\omega)$ be a compact symplectic manifold and let $M=S\times\mathbb{R}$, with the Poisson structure $\pi$ whose symplectic leaves are $(S\times \{t\},\omega)$. Let us fix the leaf $S= S\times \{0\}$. One can check directly (or by Example \ref{ex:computations} below) that
\[ 
H^{2}_{\pi,S}(M)\cong H^1(S)\oplus H^2(S),\quad 
H^{1}_{\pi,S}(M,\nu_S)= H^{0}(S)\oplus H^{1}(S),
\]
hence these two groups can never be zero. One can also see directly that $S$ is neither strongly stable nor algebroid stable.

On the other hand, it is easy to see that the the relative Poisson cohomology is
\[ H^{2}_{\pi}(M,S)\cong H^1(S).\]
In fact, both cases ($S$ stable or unstable) may arise. 

To produce an unstable example one should, of course, choose $S$ so that 
$H^1(S)\neq 0$. A natural choice to make is $S=\Ss^{1}\times \Ss^{1}$ with the usual area form:
\[ \omega=\d\theta_1\wedge \d\theta_2.\]
Indeed, one can find on $\Ss^1\times \Ss^1\times \mathbb{R}$ the Poisson structure:
\[ \pi_{\epsilon}= \frac{\partial}{\partial \theta_2}\wedge \frac{\partial}{\partial \theta_1}+ \epsilon \frac{\partial}{\partial \theta_2}\wedge \frac{\partial}{\partial t}\]
whose leaves for $\epsilon\neq 0$ are never compact.
\end{ex}

\begin{ex}
\label{ex-stable}
Consider $M=\Ss^2\times \Ss^2 \times (-1,1)$, with the Poisson structure $\pi$ whose symplectic leaves are
$\Ss^2\times \Ss^2\times \{t\}$ endowed with the symplectic form
\[ \omega_{t}= (1-t)\omega_1+ (1+t)\omega_2,\]
where $\omega_i=\text{pr}_{i}^{*}\omega$ are the pull-backs by the projections 
$\text{pr}_i: \Ss^2\times \Ss^2\to \Ss^2$ of the area form on $\Ss^2$. For the leaf $S=\Ss^2\times \Ss^2 \times \{0\}$ we find (see Example \ref{ex:computations} below):
\[ 
H^{2}_{\pi}(M, S)= 0,\quad H^{2}_{\pi, S}(M)\cong \mathbb{R},\quad 
H^{1}_{\pi, S}(M;\nu_{S})= 0.
\]
It follows from Theorem \ref{theorem11} that $S$ is stable, and Theorem \ref{theorem33} shows that it is even algebroid stable. Notice that we cannot apply Theorem \ref{theorem22} to conclude that $S$ is strongly stable. In fact, $S$ is not strongly stable: consider the family of Poisson structures $\pi_{\epsilon}$ on $M$, with the same symplectic leaves $\Ss^2\times \Ss^2\times \{t\}$ but endowed with the symplectic form
\[ \omega_{t}^{\epsilon}=(1-t+\epsilon)\omega_1+(1+t+\epsilon)\omega_2.\]
We have $\pi_{0}=\pi$, but no symplectic leaf $(\Ss^2\times\Ss^2\times\{t\},\omega_{t}^{\epsilon})$
with $\epsilon\neq 0$ can be symplectomorphic to $(\Ss^2\times \Ss^2,\omega_{0})$. This follows by computing the symplectic volumes: 
\[ 
\vol(\Ss^2\times \Ss^2\times\{t\},\omega_{t}^{\epsilon})= (2+\epsilon)\vol(\Ss^2)\neq 2 
\vol(\Ss^2)= \vol(\Ss^2\times \Ss^2,\omega_{0}).\]
\end{ex}

\begin{ex}
\label{ex:computations}
The previous two examples fit into the following more general scheme which can be used to obtain many other interesting explicit examples. Let $M= S\times I$ where $I$ is an open interval containing the origin and assume that $(\omega_{t})_{t\in I}$ is a smooth family of symplectic forms on $S$. Then $M$ becomes a Poisson manifold
with symplectic leaves $(S\times \{t\},\omega_t)$. We look at the stability of $S= S\times \{0\}$. 

The cohomologies relevant for the stability of $S$ depend only on the topology of $S$ and on the variation of symplectic areas:
\[ \sigma:= \left.\frac{\d}{\d t}\right|_{t= 0} \omega_t\in \Omega^2(S).\]
In general, any closed two-form $\sigma\in \Omega^2(S)$ induces a cohomology $H^{\bullet}_{\sigma}(S)$ as follows:
If $\sigma$ is prequantizable, i.e., if it is the curvature 2-form of some connection on a principal $\Ss^1$-bundle $P\to S$, then $H^{\bullet}_{\sigma}(S)=H^\bullet(P)$. In the general case, one uses an algebraic model which is constructed as follows: Interpreting $\sigma$ as a chain map
\[
\xymatrix{\Omega^{\bullet}(S)\ar[r]^{\sigma \wedge - }& \Omega^{\bullet + 2}(S)}, 
\]
we form its mapping cone, which is a new complex $(C_{\sigma}(S),\d_{\sigma})$ with
\[ C_{\sigma}^{\bullet}(S)=\Omega^{\bullet}(S)\oplus \Omega^{\bullet - 1}(S),\quad \d_{\sigma}(\omega,\eta)= (\d\omega+\sigma\wedge\eta, \d\eta).\]
Its cohomology, denoted $H^{\bullet}_{\sigma}(S)$, is the algebraic model we are looking for. In fact, by a standard short exact sequence argument, one obtains a long exact sequence
\begin{align*}
&\xymatrix{ 
0\ar[r]& H^{1}(S)\ar[r]& H^{1}_{\sigma}(S) \ar[r]& H^0(S)\ar[r]^{\sigma}& H^2(S)\ar[r]&}\\
&\xymatrix{ &\ar[r]&H^{2}_{\sigma}(S)\ar[r] &H^1(S)\ar[r]^{\sigma}&H^3(S)\ar[r]& H^3_{\sigma}(S)\ar[r]&\ldots}
\end{align*}
which is an algebraic interpretation of the Gysin sequence for principal $\Ss^1$-bundles. 
Using this cohomology, one checks that the relative and the restricted Poisson cohomology complexes can be identified as
\[ \X^{\bullet}_{\pi}(M, S)\cong \Omega^{\bullet-1}(S),\quad \X^{\bullet}_{\pi, S}(M)\cong C_{\sigma}^{\bullet}(S)\cong \X^{\bullet}_{\pi,S}(M;\nu_S).\]
Hence, we conclude that:
\begin{enumerate}[(i)]
\item $H^{2}_{\pi}(M, S)= 0$ if and only if $H^1(S)= 0$.
\item $H^{2}_{\pi,S}(M)= 0$ if and only if $H^2(S)= \mathbb{R}\sigma$ and cup-product by $\sigma$ determines a injective morphism from $H^1(S)$ into $H^3(S)$.
\item $H^{1}_{\pi,S}(M;\nu_S)= 0$ if and only if $H^1(S)= 0$ and $\sigma$ is non-zero in cohomology.
\end{enumerate}
\end{ex}

\section{Outline of the proofs}  
\label{sec:outline}

We now give an outline of the proof of Theorem \ref{theorem11} (the proofs of the other two theorems are variations of the same idea). The first step, consists of a simple but crucial remark. Each of the other steps correspond to each of the sections to follow.

\textbf{Step 1:} The first step exploits the cosymplectic nature of small transversals to a symplectic leaf. Recall that a \textbf{cosymplectic submanifold} of $(M, \pi)$ is any immersed submanifold $N\hookrightarrow M$ with the property that $N$ is transverse to the symplectic leafs of $M$ and at every $x\in N$ the tangent space $T_xN$ intersects the tangent space to the leaf in a symplectic subspace. Each such submanifold carries an ``induced'' Poisson structure $\pi_N$ whose symplectic leafs are the connected components of the intersections of $N$ with the symplectic leaves of $M$. Note that the cosymplectic condition is an open condition both for $x\in M$ and for $\pi$ in the space of Poisson structures with the $C^{\kappa}$-topology. For a detailed discussion of this kind of submanifolds and its properties we refer to \cite[Section 8]{CrFe3}.

Let us recall now that one calls a Poisson structure $\theta$ on the total space of a bundle $p:E\to S$ \textbf{horizontally non-degenerate} if the fibers of $E$ are cosymplectic with respect to $\theta$. Let $S$ be an embedded symplectic leaf $S$ of a Poisson manifold $(M,\pi)$. Our first remark is that if we are interested on the behavior around $S$ of Poisson structures which are close to $\pi$, then we can restrict our attention to the case where $M=E$ is a vector bundle over $S$ and to Poisson structures $\theta\in\X^2(E)$ which are horizontally non-degenerate. Indeed, we just take $E=\nu(S)$, which can be identified with an arbitrarily small open neighborhood of $S$ in $M$ via an exponential map construction. Since the fibers of $E$ are transverse to $S$, $\pi$ satisfies the cosymplectic condition at points of $S$. Hence, for a small enough tubular neighborhood $\pi$ and all nearby $\theta$ will be horizontally nondegenerated.

\textbf{Step 2:} The second step is to identify the algebraic structure underlying horizontally non-degenerate Poisson structures. This is a certain bigraded algebra $\Omega^{\bullet,\bullet}_E$ which is endowed with a graded Lie bracket $[\cdot, \cdot]$. We observe that horizontally non-degenerate bivector fields $\theta$ correspond to elements in the algebra which break into three homogeneous components, defining a triple $(\theta^{\textrm{v}},\Gamma_{\theta},\Ff_{\theta})$. The condition for a bivector to be Poisson (i.e., $[\theta,\theta]=0$) translates into certain equations involving the homogeneous components, which we call ``the structure equations'' of the triple. They can be expressed in an elegant and efficient way using the graded Lie bracket. Moreover, this algebraic structure also allows us to define operations of ``restriction'' and ``linearization'' along a section, which will be crucial for the proofs. Along the way, we obtain a reformulation of Vorobjev's d!
 escription of horizontally non-degenerate Poisson structures (see  \cite{Vor}).

\textbf{Step 3:} In this step we use sections $s\in \Gamma(E)$ to produce symplectic leaves: any $s\in \Gamma(E)$ defines a submanifold of $E$ diffeomorphic to $S$, namely its graph
\[ \Graph(s):= \{ s(x): x\in S\}\subset E.\]
We will answer the following question: given a horizontally non-degenerate Poisson structure $\theta$, with associated triple $(\theta^{\textrm{v}}, \Gamma_{\theta},\Ff_{\theta})$, when
does a section $s\in \Gamma(E)$ define a symplectic leaf $\Graph(s)$ of $\theta$? We will see that one can restrict the first two components of the triple along $s$  to obtain an element
\[ c_{\theta,s}:= (\theta^{\textrm{v}}|_{s}, \Gamma_{\theta}|_{s}) \in \Omega^{1}_{S}\]
whose vanishing is equivalent to $\Graph(s)$ being a symplectic leaf of $\theta$.

\textbf{Step 4:} In this step we determine the p.d.e.~that $c_{\theta, s}$ satisfies arising from 
the condition that $\theta$ is Poisson. We will see that this p.d.e.~is just the ``linearization'' along $s$ of the structure equations of the geometric triple, and that it takes the form
\[ \overline{\d}_{\theta,s}(c_{\theta,s})= 0,\]
where $\overline{\d}_{\theta, s}$ is an operator obtained by linearizing the Poisson cohomology operator
around $s$. When $\theta=\pi$ and $s$ is the zero section, $\overline{\d}_{\pi,0}$ will be just the
differential on the complex computing the relative Poisson cohomology $H^\bullet_{\pi}(M,S)$, 
while the kernel of $\overline{\d}_{\pi,0}$ on $\Gamma(E)$ is precisely the parameter space $\Gamma_{\text{flat}}(\nu^{0}_{S})$ which appears in the statement of the theorem.

\textbf{Step 5:} Suppose $\theta$ is close to $\pi$. By Step 3, in order to produce symplectic leaves of $\theta$ diffeomorphic to $S$, we need to look for sections $s$ such that 
$c_{\theta, s}= 0$. We will find them among the critical points of the functional
\[ \Phi_{\theta}(s)= ||c_{\theta,s}||^{2},\]
for some convenient Sobolev norm $||\cdot||$. One must therefore make the appropriate choices of Hilbert spaces,
and check that the operations of restriction and linearization along sections still make sense and are continuous. In this step, we take care of the necessary analytic machinery.

\textbf{Step 6:} When $\theta=\pi$, $s=0$ is a degenerate critical point of the functional $\Phi_\theta$. Also, any section in $\Gamma_{\text{flat}}(\nu^{0}_{S})$ is also a zero of $\Phi_\pi$. Hence, we
restrict $\Phi_\pi$ to the orthogonal complement $W$ of $\Gamma_{\text{flat}}(\nu^{0}_{S})$ in $\Gamma(E)$. Then $s=0$ becomes a strongly non-degenerate critical point of $\Phi_{\pi}$, hence nearby $\theta$'s will produce functionals $\Phi_{\theta}$ which admit critical points parametrized by $\Gamma_{\text{flat}}(\nu^{0}_{S})$. In this final step we will show that any small enough critical point $s$ of $\Phi_{\theta}$ must satisfy $c_{\theta, s}= 0$. 

First, the critical point condition is equivalent to
 \[ \overline{\d}_{\pi, 0} \overline{\d}_{\theta, s}^{*}(c_{\theta,s})=0.\] 
This, combined with the equation already established in Step 4, shows that $c_{\theta, s}$ is in the kernel of the operator
\[ 
\lap_{\theta,s}:=\overline{\d}_{\pi,0}\overline{\d}_{\theta,s}^{*}(c_{\theta,s})+\overline{\d}_{\pi,0}^*\overline{\d}_{\theta,s}(c_{\theta, s}),
\]
defined on the appropriate Sobolev spaces. However, when $\theta=\pi$ and $s=0$, we show that
the ellipticity of the complex computing $H^\bullet_{\pi}(M, S)$ implies that the operator $\lap_{\pi,0}$ is invertible. It follows that $\lap_{\theta, s}$ must be an isomorphism for $s$ small enough and $\theta$ close to $\pi$. Hence, we must have $c_{\theta,s}= 0$.

\section{Algebraic-geometric framework}                  %
\label{sec:hor:non-degenerate}                            %



In this section we fix a vector bundle $p:E\to S$ and we focus our attention on horizontally non-degenerate Poisson bivectors $\theta\in\X^2(E)$. We introduce a bigraded algebra $\Omega^{\bullet,\bullet}_E$, endowed with a graded Lie bracket $[\cdot, \cdot]$, that allows for an efficient description of horizontally non-degenerate Poisson structures in terms of triples. In the next two sections we will use this algebraic structure to define operations of ``restriction'' and ``linearization'' along a section. The description of horizontally non-degenerate Poisson structures in terms of geometric triples appeared first in Vorobjev's (see \cite{Vor}), but he did not make explicit use of the graded Lie algebra $\Omega_E$, which plays a central role in our approach.

\subsection{The bigraded algebra $\Omega_E$}
The $(q,p)$-component of $\Omega_E$ is obtained as follows. We consider the vector bundle 
$\wedge^pT^*S\otimes \wedge^q E$ and take its pullback under the projection $p:E\to S$. Then $\Omega^{q,p}_E$ is the space of sections of this pullback:
\[ \Omega^{q,p}_E:=\Gamma(p^*(\wedge^pT^*S\otimes \wedge^q E)). \]
It will also be convenient to consider the ``restriction'' of $\Omega_E$ to $S$, denoted $\Omega_S$, which is just the space of differential forms in $S$ with values in $\wedge E$. Its $(q,p)$-component is then:
\[ \Omega_{S}^{q,p}=\Omega^p(S;\wedge^q E).\]
Note that, algebraically, 
\[ \Omega_{E}= \Omega_{S}\otimes_{C^{\infty}(S)} C^{\infty}(E).\]
Therefore, we can also view $\Omega_{E}^{q,p}$ as the space of $p$-forms on $S$ with values in 
$q$-vertical multivector fields on $E$:
\[ \Omega_{E}^{q,p}= \Omega^p(S)\otimes_{C^{\infty}(S)} \X^{q}(\Ver),\]
where $\Ver=\Ker \d p\subset TE$ denotes the vertical subbundle and $\X^q(\Ver)$ the vertical multivector fields: $\X^q(\Ver):=\Gamma(\wedge^q \Ver)\subset\X^q(E)$.

The last description of $\Omega_{E}$ shows that this space is naturally endowed with a bracket
\[ [\cdot, \cdot]_{\Omega_E}: \Omega_{E}^{q,p}\times \Omega_{E}^{q',p'}\rmap \Omega_{E}^{q+q'-1,p+p'}\]
which arises from the wedge product of forms on $S$ and the Schouten bracket on $\X^{\bullet}(\Ver)$: if $\al_1\in\Omega_{E}^{q,p}$ and $\al_2\in\Omega_{E}^{q',p'}$ we set 
\begin{multline*}
[\al_1,\al_2]_{\Omega_E}(X_1,\dots,X_{p+p'}):=\\\sum_{\sigma}(-1)^{|\sigma|}
[\al_1(X_{\sigma(1)},\dots,X_{\sigma(p)}),\al_2(X_{\sigma(p+1)},\dots,X_{\sigma(p+p')})].
\end{multline*}
Henceforth we will drop the subscipt $\Omega_E$ from the notation (this is consistent with the fact that we use $[~,~]$ for the Schouten bracket and that this bracket extends it).

If one considers local coordinates $(U,x^i)$ on $M$ and a local frame $(e_a)$ of $E$ over $U$, inducing fiberwise coordinates $(y^a)$ on $E$, an element $u\in \Omega^{q,p}_E$ is written as:
\begin{equation}
\label{general-u} 
u= u_I^J(x,y)~\d x^{I}\otimes\partial y_{J}
\end{equation}
for some functions $u_I^J\in C^\infty(p^{-1}(U))$. Here $I=(i_1,\ldots,i_p)$ is a multi-index of length $|I|=p$, $J=(a_1,\ldots,a_q)$ is a multi-index of length $|J|=q$, and we sum over repeated multi-indices. The elements $\d x^I$ are homogeneous of degree $(0,p)$ and defined by
\[ \d x^{I}= \d x^{i_1}\wedge\ldots\wedge \d x^{i_p},\]
while the elements $\partial y_{J}$ are homogeneous of degree $(q,0)$ and, according to our two descriptions of $\Omega_E$, can be interpreted either as a section of $\wedge^q\Ver$ or as a section of $\wedge^q E$:
\[ \partial y_{J}= \frac{\partial}{\partial y^{i_1}} \wedge \ldots \wedge \frac{\partial}{\partial y^{i_q}}= e_{i_1}\wedge \ldots \wedge e_{i_q}.\]
We will call the $u_I^J(x,y)$ the \textbf{local coefficients} of the element $u$. Notice that $\Omega_S\subset \Omega_E$ consists of those elements whose local coefficients do not depend on $y$. We will also need the subspace $\Omega_{E,\text{lin}}\subset\Omega_E$ consisting of those elements whose local coefficients $u_I^J(x,y)$ are linear in $y$. In a more invariant form, they can be identified with 
\[ 
\Omega_{E,\text{lin}}=\Omega_{S}\otimes_{C^{\infty}(S)}\Gamma(E^*)=\Omega(S;\wedge E\otimes E^*),
\]
where we think of a section of $E^*$ as a fiberwise linear function on $E$. 

\begin{rem}
\label{purely-algebraic}
Notice that $[\cdot, \cdot]$ is trivial on $\Omega_S$, induces a graded Lie algebra structure on $\Omega_{E, \text{lin}}$, and makes $\Omega_S$ into a representation of $\Omega_{E, \text{lin}}$: for each $v\in \Omega_{E,\text{lin}}$ we have
\[ \ad_{v}: \Omega_S\to \Omega_S,\quad w\mapsto [v, w].\]
Note also that, the Lie bracket induced on the subspace $\Omega_{E, \text{lin}}$ and its representation $\Omega_S$, is purely algebraic: it involves no derivatives and takes place fiberwise.
\end{rem}

Connections also fit naturally in this framework. Let $\Gamma$ be a connection on $E$, i.e., a horizontal distribution $\Hor_{\Gamma}\subset TE$ with the property that 
\[ TE= \Hor_{\Gamma}\oplus \Ver. \]
The corresponding horizontal lift operation $\hor=\hor_{\Gamma}$,
\[ \hor_{e}: T_{p(e)}S\rmap T_eE,\quad (e\in E)  \]
associates to each vector $X\in T_{p(e)}S$ the unique horizontal vector $\hor_{e}(X)\in \Hor_{\Gamma, e}$ that projects to $X$ under $p$. If one introduces local coordinates $(x^i, y^a)$ as above, the connection $\Gamma$ is determined by functions $\Gamma_{i}^{j}(x,y)$ such that:
\begin{equation}
\label{general-c}
 \hor_{\Gamma}(\frac{\partial}{\partial x^i})= \frac{\partial}{\partial x^i}+\Gamma_{i}^{a}(x, y)\frac{\partial}{\partial y^a}.
\end{equation}
We call $\Gamma_{i}^{a}(x, y)$ the local coefficients of the connection $\Gamma$.

A connection can be interpreted in terms of the graded algebra $\Omega_E$ as an operator
\[ \d_{\Gamma}: \Omega_{E}^{\bullet, \bullet}\rmap \Omega_{E}^{\bullet, \bullet+1},\]
defined by the following Koszul type formula:
\begin{multline*} 
(\d_{\Gamma}\omega)(X_1, \ldots , X_{p+1})= \\
\sum_{i}(-1)^{i+1} \Lie_{\textrm{hor}(X_i)}(\omega(X_1, \ldots , \widehat{X_i}, \ldots , X_{p+1}))+\\
+\sum_{i< j} (-1)^{i+j}\omega([X_i, X_j], \ldots , \widehat{X_i}, \ldots, \widehat{X_j}, \ldots , X_{p+1}),
\end{multline*}
for $\omega\in \Omega^{q,p}_{E}$, $X_1, \ldots , X_{p+1}\in \X(S)$.

The operator $\d_{\Gamma}$ usually fails to be a differential. This failure is measured by the curvature of the connection $\Gamma$, which is usually defined by the expression
\[ 
\Omega_{\Gamma}(X, Y)= [\textrm{hor}(X), \textrm{hor}(Y)]- \textrm{hor}([X, Y])\in \X(\Ver),
\]
for $X,Y\in\X(S)$. In our context, this tensor can be viewed as an element in on our bigraded algebra $\Omega_{\Gamma}\in \Omega_{E}^{1,2}$, and we have:
\begin{equation}
\label{eq:curv:connection}
\d_{\Gamma}^2=\ad_{\Omega_{\Gamma}}.
\end{equation}

Linear connections fit into this framework as follows. A linear connection $\Gamma$ on $E$ determines (and is determined by) a covariant derivative operator $\nabla:\X(S)\otimes\Gamma(E)\to\Gamma(E)$. In local coordinates, $\Gamma$ is linear iff the coefficients $\Gamma^{i}_{a}(x, y)$ are linear in $y$:
\[ \Gamma_{i}^{a}(x, y)= \Gamma^{a}_{i, b} (x) y^b,\]
and the associated covariant derivative operator is then given by:
\[ \nabla_{\frac{\partial}{\partial x_i}}~e_a= \Gamma_{i, a}^{b} e_b.\]
Alternatively, one can use the bigraded algebra $\Omega_E$. A connection $\Gamma$ is linear 
if and only if $\d_{\Gamma}$ preserves $\Omega_{E,\textrm{lin}}$. In this case, $\d_{\Gamma}$ also preserves $\Omega_S$ and its restriction to $\Omega_{S}^{1,\bullet}=\Omega^{\bullet}(S;E)$ is an operator
\[ \d_{\Gamma}: \Omega^{\bullet}(S;E)\to \Omega^{\bullet+1}(S;E)\]
which is just another way of viewing the covariant derivative associated to the linear connection. 

\subsection{The geometric triples $(\theta^{\text{\rm v}}, \Gamma_{\theta}, \mathbb{F}_{\theta})$}
After these preliminaries about connections, we can now return to the study of horizontally non-degenerate Poisson structures. Recall that a bivector field $\theta\in\X^2(E)$ is 
horizontally nondegenerated if
\[ TE= \Ver\oplus(\Ver)^{\perp,\theta}.\] 
Therefore, we have a connection $\Gamma_\theta$ with horizontal distribution defined by
\[ \Hor_{\theta}:= (\Ver)^{\perp, \theta}.\]
Next, observe that $\theta$ ``restricts'' to the fibers of $E$, since these fibers are cosymplectic submanifolds. We denote this ``restriction'' by $\theta^{\textrm{v}}\in \Omega_{E}^{2,0}$. Finally, the cosymplectic condition also implies that the restriction of $\theta$ to $\Hor_{\theta}$ is non-degenerate hence it defines an element in $\wedge^2\Hor_{\theta}^*$. The differential $\d p:TE\to TS$ induces an isomorphism $\Hor_{\theta}\to TS$ and the image of $(\theta|_{\Hor_\theta})^{-1}$ under this map is an element $\mathbb{F}_{\theta}\in\Omega_{E}^{0,2}$. In this way, to a horizontally non-degenerate bivector field $\theta$ we associate a triple $(\theta^{\textrm{v}}, \Gamma_{\theta}, \mathbb{F}_{\theta})$. Conversely, it is not difficult to see that the bivector $\theta$ can be reconstructed from the resulting triple. 

Let us now ask what does the condition for $\theta$ to be Poisson, i.e., $[\theta,\theta]=0$,
corresponds to, in terms of the associated triple $(\theta^{\textrm{v}},\Gamma_{\theta}, \mathbb{F}_{\theta})$. The answer is given by the following result due to Vorobjev \cite{Vor}:

\begin{thm} 
\label{Vorobjev} 
There is a 1-1 correspondence between horizontally non-degenerate Poisson structures $\theta$ on a vector bundle $p:E\to S$ and triples $(\theta^{\text{\rm v}},\Gamma_\theta,\mathbb{F}_\theta)$,
consisting of elements $\theta^{\text{\rm v}}\in \Omega_{E}^{2,0}$, $\mathbb{F}_\theta\in \Omega_{E}^{0,2}$ and a connection $\Gamma_\theta$ on $E$, satisfying the structure equations:
\begin{enumerate}[(i)]
\item $[\theta^{\text{\rm v}}, \theta^{\text{\rm v}}]= 0$ (i.e., $\theta^{\textrm{v}}$ is a vertical Poisson structure).
\item $\d_{\Gamma_\theta}\theta^{\text{\rm v}}= 0$ (i.e., parallel transport along $\Gamma$ preserves $\theta^{\text{\rm v}}$).
\item $\d_{\Gamma_\theta}\Ff=0$.
\item $\Omega_{\Gamma_\theta}=[\Ff_\theta,\theta^{\text{\rm v}}]$.
\end{enumerate}
\end{thm}

There is another way of looking at the structure equations in terms of the graded algebra $\Omega_E$. For that, we observe that:
\begin{itemize}
\item The vertical bivector $\theta^{\textrm{v}}\in\Omega_{E}^{2,0}$ induces an operator
\[ \d^{1,0}_{\theta}:=\ad_{\theta^{\textrm{v}}}: \Omega_{E}^{q,p}\to \Omega_{E}^{q+1,p};\]
\item The connection $\Gamma_\theta$ induces an operator 
\[ \d^{0,1}_{\theta}:=\d_{\Gamma_\theta}:\Omega_{E}^{q,p}\to\Omega_{E}^{q,p+1}; \]
\item The 2-form $\Ff_{\theta}\in\Omega_{E}^{2,0}$ induces also an operator 
\[ \d^{-1,2}_{\theta}:=\ad_{\Ff_\theta}: \Omega_{E}^{q,p}\rmap \Omega_{E}^{q-1,p+2}.\]
\end{itemize}
We also introduce the total operator (which rises the total degree by $1$!):
\begin{equation}
\label{total-op} 
\d_{\theta}:= \d^{1,0}_{\theta}+\d^{0,1}_{\theta}+\d^{-1,2}_{\theta}:\Omega_E\to\Omega_E,
\end{equation}
Then we have:

\begin{prop}
If $\theta\in\X^2(E)$ is a horizontally nondegenerated Poisson tensor on a vector bundle $p:E\to S$ then $\d_\theta^2=0$. The complex $(\Omega_E,\d_\theta)$ is isomorphic to the Poisson cohomology complex $(\X(E),\d_\theta)$.
\end{prop}

\begin{proof}
Note that the equation $\d_{\theta}^2=0$ simply says that the operators $\d_{\Gamma_\theta}$, $\ad_{\theta^{\textrm{v}}}$ and $\ad_{\Ff_\theta}$ all commute, with the exception of:
\begin{equation}
\label{eq:commute:dif}
[\d_{\Gamma_\theta},\d_{\Gamma_\theta}]= -[\ad_{\Ff_\theta}, \ad_{\theta^{\textrm{v}}}].
\end{equation}
The first 3 structure equations from Theorem \ref{Vorobjev} yield the commutation of these operators, while the last structure equation together with \eqref{eq:curv:connection}, show that \eqref{eq:commute:dif} holds.

Also, one checks easily that the decomposition
\[ TE= \textrm{Hor}_{\theta}\oplus \Ver \cong T^*S\oplus \Ver ,\]
induces an isomorphism (which depends on $\theta$):
\[ \X^{k}(E)\cong \oplus_{p+q= k} \Omega^{p, q}_{E},\]
and the Poisson differential on $\X(E)$ corresponds to the differential $\d_\theta$ on $\Omega_E$.
\end{proof}

\begin{rem}
The proposition states that $\theta$ Poisson implies that $\d_\theta^2=0$ on $\Omega_E$. However, the converse is not true. This is because condition (iv) in Theorem \ref{Vorobjev} implies that \eqref{eq:commute:dif} holds, but the converse is not true.
\end{rem}

It is convenient to picture the complex $(\Omega_E,\d_\theta)$, using the decomposition \eqref{eq:commute:dif}, as a diagram:
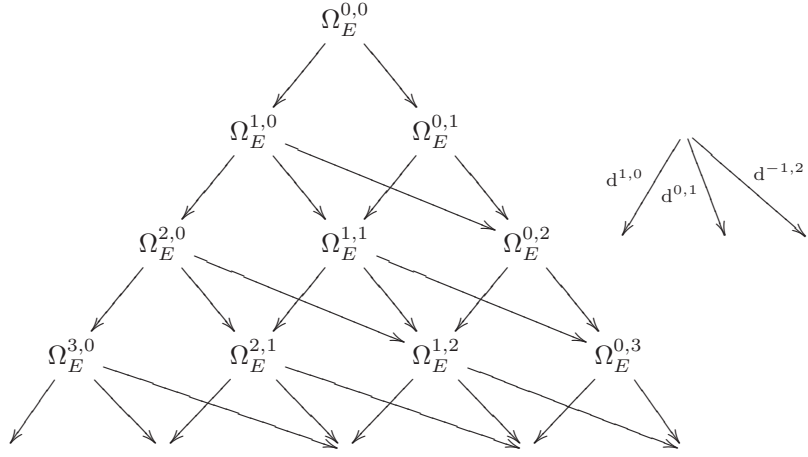
\[
\xymatrix@C=10pt{
&&&&\Omega^{0,0}_E\ar[dr]\ar[dl]&&&&&&&\\
&&&\Omega^{1,0}_E\ar[drrr]\ar[dr]\ar[dl]& &\Omega^{0,1}_E\ar[dr]\ar[dl]&&&\ar[drrr]^{\d^{-1,2}}\ar[dr]_{\d^{0,1}}\ar[dl]_{\d^{1,0}}&&&\\
&&\Omega^{2,0}_E\ar[drrr]\ar[dr]\ar[dl]&  &\Omega^{1,1}_E\ar[drrr]\ar[dr]\ar[dl]&  &\Omega^{0,2}_E\ar[dr]\ar[dl]&&&&&\\
&\Omega^{3,0}_E\ar[drrr]\ar[dr]\ar[dl]& &\Omega^{2,1}_E\ar[drrr]\ar[dr]\ar[dl]& &\Omega^{1,2}_E\ar[drrr]\ar[dr]\ar[dl]& &\Omega^{0,3}_E\ar[dr]\ar[dl]&&&&\\
&&&&&&&&&&&}
\]
Note also that, if $S$ is a symplectic leaf, then under the isomorphism $\X(E)\cong\Omega_E$ 
the restricted and relative complexes are mapped to the complexes:
\[
\X_{S}^{k}(E)\cong \Bigoplus_{p+q= k} \Omega^{q,p}_{S},\qquad
\X^{k}(E, S)\cong \Bigoplus_{\substack{p+q=k\\ q\geq 1}} \Omega^{q,p}_{S}.
\]
The differential of the restricted complex $\Omega_S$ is the restriction of the differential of the complex $\Omega_E$, so we also have a decomposition as in \eqref{eq:commute:dif} and a similar diagram:
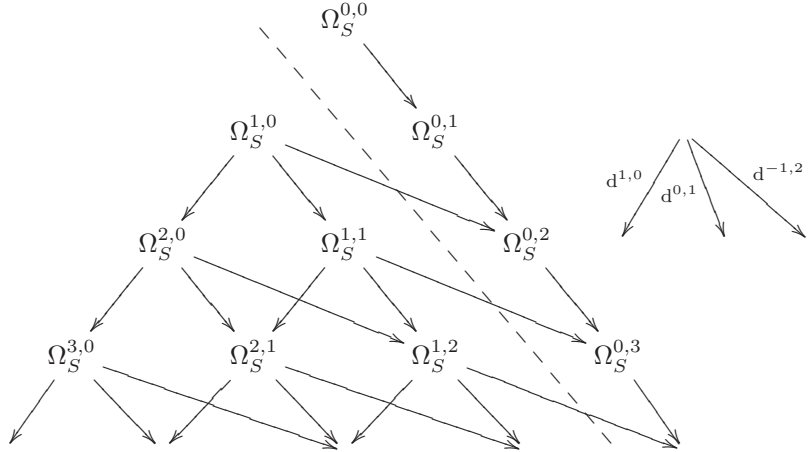
\[
\xymatrix@C=10pt{
&&&&\Omega^{0,0}_S\ar[dr]&&&&&&&\\
&&&\Omega^{1,0}_S\ar[drrr]\ar[dr]\ar[dl]& &\Omega^{0,1}_S\ar[dr]&&&\ar[drrr]^{\d^{-1,2}}\ar[dr]_{\d^{0,1}}\ar[dl]_{\d^{1,0}}&&&\\
&&\Omega^{2,0}_S\ar[drrr]\ar[dr]\ar[dl]&  &\Omega^{1,1}_S\ar[drrr]\ar[dr]\ar[dl]&  &\Omega^{0,2}_S\ar[dr]&&&&&\\
&\Omega^{3,0}_S\ar[drrr]\ar[dr]\ar[dl]& &\Omega^{2,1}_S\ar[drrr]\ar[dr]\ar[dl]& &\Omega^{1,2}_S\ar[drrr]\ar[dr]\ar[dl]& &\Omega^{0,3}_S\ar[dr]&&&&\\
&&&&&&&\ar@{--}@<0ex>[uuuullll]&&&&&}
\]
The top diagonal row in this diagram is just the de Rham complex of $S$, which is a subcomplex of $\Omega_S$. The  quotient complex is $\overline{\Omega}_{S}$ and coincides with the part of the diagram below the dotted line; we denote by $\overline{\d}_{\theta, s}$ its associated operator.
Actually, the isomorphism between $\Omega_{S}$ and $\X_{S}(E)$ induces an identification between 
the short exact sequence defining the relative complex $\X(E, S)$ and the one associated to
the inclusion of $\Omega(S)$ inside $\Omega_{S}$:
\[ \xymatrix{
0\ar[r] & \Omega^\bullet(S)\ar[r] \ar[d]^-{\textrm{Id}}& \X_{S}^\bullet(E) \ar[r]\ar[d]^-{\sim}& \X^\bullet(E, S)\ar[r]\ar[d]^-{\sim}& 0.\\
0\ar[r] & \Omega^\bullet(S)\ar[r] & \Omega_{S} \ar[r]& \overline{\Omega}_{S}\ar[r]& 0
}.
\]
Hence, $\overline{\Omega}_{S}$ can be identified with the relative complex $\X(E,S)$.

\section{Restricting to sections}                  %
\label{sec:leaves:sections}                               %

In the previous section we have studied horizontally non-degenerate Poisson structures $\theta$
on the total space of a vector bundle $p: E\to S$ and explained that they correspond to
certain triples $(\theta^{\textrm{v}},\Gamma_\theta, \Ff_\theta)$ which interact inside a graded algebra $\Omega_E$. In this section, we place ourselves in the same framework and we look at the condition
that a section $s\in \Gamma(E)$ has to satisfy so that $\Graph(s)\subset E$ is a symplectic leaf of $\theta$.

The relevant algebraic operation for solving this question is \textbf{``the restriction operation''} from $\Omega_E$ to $\Omega_S$ induced by a section $s\in\Gamma(E)$. For functions, this is just the map: 
\[ s^*:C^{\infty}(E)\to C^{\infty}(S),\quad f\mapsto f\circ s. \]
Taking the tensor product with $\Omega_S$ we obtain the restriction map
\[ \Omega_E\rmap \Omega_S,\quad \omega\mapsto \omega|_s.\]
This notation is justified since, in local coordinates, this restriction map sends an element 
(\ref{general-u}) with coefficients $u_{I, J}(x, y)$ into a new element with coefficients 
$u_{I,J}(x,s(x))$. 

We also need the operation of restricting a connection $\Gamma$ on $E$ along a section $s$. This 
is a bit more subtle and will be further explained in the next section. For now, we take the following 
ad-hoc description: Given a connection $\Gamma$ on $E$ and a section $s\in \Gamma(E)$, the
\textbf{restriction of $\Gamma$ to $s$} is defined as the $E$-valued $1$-form on $S$, denoted $\Gamma|_{s}\in \Omega_{S}^{1, 1}$, given by
\[ \Gamma|_{s}(X_x):=\hor_{s(x)}(X_x)-(\d_x s)(X_x)\in\Ver_{s(x)}\cong E_x.\]
In local coordinates, if $\Gamma_i^a(x,y)$ are the coefficients (\ref{general-c}) of $\Gamma$ , then 
\begin{equation}
\label{Gamma-to-s} 
\Gamma|_{s}=\left(\Gamma_{i}^{a}(x, s(x))-\frac{\partial s^a}{\partial x^i}\right) \d x^i\otimes\partial_{y^a}.
\end{equation}

Assume now that one is given a horizontally nondegenerate Poisson structure $\theta$ on $E$, with associated
triple $(\theta^{\textrm{v}},\Gamma_\theta, \Ff_\theta)$. For each section $s\in \Gamma(E)$ we define
\[ c(\theta, s):=\theta^{\textrm{v}}|_{s} + \Gamma_{\theta}|_{s} \in \Omega_{S}^{2,0}\oplus\Omega_{S}^{1,1}.\]

\begin{prop}
\label{criteria-leaf} 
Given a horizontally non-degenerate Poisson structure $\theta$ on the vector bundle $p:E\to S$, a section $s\in \Gamma(E)$ defines a symplectic leaf of $\theta$ if and only if $c(\theta,s)=0$.
\end{prop}

\begin{proof}
Recall that the horizontal non-degeneracy of $\theta$ at $s(x)$ means that:
\[ T_{s(x)}E=\Ver_{s(x)} \oplus \Hor_{s(x)}=\Ver_{s(x)} \oplus\, \theta^\sharp_{s(x)}(\Ver_{s(x)})^0.\]
Now, given a section $s\in \Gamma(E)$, observe that $\Graph(s)$ is a symplectic leaf of $\theta$ if and only if for every $x\in S$ we have:
\[ \theta^\sharp_{s(x)}(T^*_{s(x)}E)=T_{s(x)}(\Graph(s)). \]
Hence, we see that this holds iff we have:
\[ \Ker\theta^\sharp_{s(x)}=(\Hor_{s(x)})^0,\quad \theta^\sharp_{s(x)}(\Ver_{s(x)})^0=T_{s(x)}(\Graph(s)).\]
The first condition says that $\theta^{\textrm{v}}|_{s}=0$, while the second condition says that $\Gamma_{\theta}|_{s}=0$. So, together, they are are equivalent to $c(\theta,s)=0$.
\end{proof}

\section{Linearizing along sections}                      %
\label{sec:linear:approx}                               %

As in the previous sections, our scenario is a vector bundle $p: E\to S$, where
we study horizontally non-degenerate Poisson structures $\theta\in\X^2(E)$. In the previous two sections, 
we have interpreted such Poisson structures as geometric triples $(\theta^{\textrm{v}},\Gamma_\theta,\Ff_\theta)$ and we have discussed their ``restrictions'' (read: ``$0$-th order approximations'') along sections $s\in\Gamma(E)$.
We will now study higher order approximations.

\subsection{Linearization along sections}
The notion of linearization along sections of $E$ is more or less obvious in local coordinates. However,
for our purpose it is much more convenient to have a global description, which is independent of such choices.
In order to formulate this, we first introduce an algebraic operation which governs this process. 

The dilatation and translation operators combine so that, for each $t\in \mathbb{R}$ and $s\in \Gamma(E)$, 
one has an affine bundle map
\[ a_{t}^{s}: E\to E,\quad e\mapsto te+ s(p(e)).\]
This induces a map on $\Omega_E$ which we still denote by the same letter:
\[ a_{t}^{s}: \Omega_{E}\to  \Omega_{E}.\]
If we think of $\Omega_E$ as the tensor product $\Omega_S\otimes C^{\infty}(E)$,
then $a_{t}^{s}$ only acts on the second component. The operation we are interested in
is obtained by rescaling: for $t\not=0$ we set
\begin{equation}
\label{eq:rescale}
\phi_{t}^{s}:= \frac{1}{t} a_{t}^{s}: \Omega_E\to \Omega_E.
\end{equation}
The reason for the rescaling will be clear in the sequel. In local coordinates $(x,y)$, one sees that the effect of the operation $\phi_{t}^{s}$ is very simple: 
\[ 
u=u_{I,J}(x,y)~\d x^I\otimes\partial y_J\longmapsto 
\phi_{t}^{s}(u)=\frac{1}{t}u_{I,J}(x,ty+s(x))~\d x^I\otimes\partial y_J.
\]
The following properties of $\phi_{t}^{s}$ should now be clear:
\begin{itemize}
\item $\phi_{t}^{s}$ \textbf{preserves the bracket} $[\cdot, \cdot]$. In particular,
 for each $u\in \Omega_E$,
 \[ \phi_{t}^{s}(\ad_u)=\ad_{\phi_{t}^{s}(u)},\]
 where, for any operator $D$ acting on $\Omega_E$, we set 
 \[ \phi_{t}^{s}(D):= \phi_{t}^{s}\circ D\circ (\phi_{t}^{s})^{-1}.\]
\item $\phi_{t}^{s}$ \textbf{acts on connections}: for any connection $\Gamma$ on $E$, there is a unique
 connection $\phi_{t}^{s}\Gamma$ satisfying:
 \[ \phi_{t}^{s}(\d_{\Gamma})=\d_{\phi_{t}^{s}\Gamma}.\]
 In local coordinates, a connection with local coefficients $\Gamma_{i}^{a}(x, y)$ is transformed to a new 
 connection with local coefficients:
 \[ 
 (\phi_{t}^{s}\Gamma)_{i}^{a}(x,y)= 
 \frac{1}{t}\left(\Gamma_{i}^{a}(x,ty+s(x)-\frac{\partial s^a}{\partial x_i}(x))\right).
 \] 
\end{itemize}

Note also that, using $\phi_t:=\phi_{t}^{0}$, the subspace $\Omega_S$ of $\Omega_E$ can be characterized as:
\begin{equation}
\label{OmegaS} 
\Omega_S=\{u\in \Omega_E: \phi_t(u)=\frac{1}{t}u,~\forall t\not=0\},
\end{equation}
while the subspace $\Omega_{E,\text{lin}}$ can characterized as:
\begin{equation}
\label{OmegaElin} 
\Omega_{E,\text{lin}}=\{u\in \Omega_E: \phi_t(u)=u,~\forall t\not=0\}.
\end{equation}
On the other hand, for an arbitrary $u\in\Omega_E$, one has
\[ \lim_{t\to 0}~t\,\phi_{t}^{s}(u)=u|_{s}.\]
This is an equality of sections of a bundle over $S$, so this limit is fiberwise and
so there is no issues about topologies. We can interpret this as saying that the formal development of $\phi_{t}^{s}(u)$ in powers of $t$ starts with the term $t^{-1}$ with coefficient $u|_s$. 

Let us consider the next term in the formal development of $\phi_{t}^{s}(u)$ in powers of $t$:
\begin{equation}
\label{develop-u} 
\phi_{t}^{s}(u)\sim \frac{1}{t} u|_{s}+ \lins u + o(t).
\end{equation}
More precisely, we set:

\begin{defn} 
For $s\in \Gamma(E)$ and $u\in \Omega_E$, the element $\lins u\in\Omega_{E,\lin}$ defined by:
\[ \lins u= \lim_{t\to 0}~\frac{t\,\phi_{t}^{s}(u)-u|_{s}}{t}, \]
is called the \textbf{linearization of $u$ along $s$}.
\end{defn}

In order to see that $\lins u$ does lie in $\Omega_{E,\lin}$ one can proceed formally using \eqref{OmegaElin} or just use local coordinates: an element $u\in \Omega_E$, with local coefficients $u_{I,J}(x,y)$, has a linearization $\lins u$ which has local coefficients:
\[ (\lins u)_{I, J}= \frac{\partial  u_{I,J}}{\partial y^b}(x,s(x))y^{b}.\]

Finally, a similar discussion applies to connections, so we describe this briefly. First of all, the restriction $\Gamma|_{s}$ introduced in the previous section is the first coefficient in the
formal development of $\phi_{t}^{s}\Gamma$:
\[ \phi_{t}^{s}\Gamma\sim \frac{1}{t} \Gamma|_{s}+ \lins\Gamma + o(t).\]
This development should be understood at the level of operators acting on $\Omega_E$:
\[ \d_{\phi_{t}^{s}\Gamma}\sim \frac{1}{t} \ad_{\Gamma|_{s}}+ \d_{\lins\Gamma} + o(t) .\]
Therefore, we set:

\begin{defn} 
Given a section $s\in \Gamma(E)$ and a connection $\Gamma$ on $E$, the linear connection
$\lins\Gamma$ defined by:
\[ \d_{\lins\Gamma}= \lim_{t\to 0}~\frac{t\,\d_{\phi_{t}^{s}\Gamma}-\ad_{\Gamma|_{s}}}{t} \]
is called the \textbf{linearization of $\Gamma$ along $s$}.
\end{defn}

In order to see that $\lins\Gamma$ is indeed a linear connection, we can use local coordinates: 
if $\Gamma_{i}^{a}(x, y)$ are the local coefficients of $\Gamma$, then the local coefficients
of the linearization $\lins \Gamma$ are:
\[  (\lins\Gamma)^{i}_{a}=\frac{\partial  \Gamma^{i}_{a}}{\partial y^b} (x,s(x))y^{b} .\]

\begin{rem}
Note that since both $\lins u$ and $\lins \Gamma$ are linear, operators of type $\ad_{\lins u}$ or $\d_{\lins \Gamma}$ map $\Omega_S$ into itself.
\end{rem}

\subsection{Linearizing Poisson structures around sections}
Let $\theta$ be a horizontally non-degenerate Poisson structure, with associated geometric triple
$(\theta^{\textrm{v}},\Gamma_\theta, \Ff_\theta)$. For any section $s\in \Gamma(E)$ we define the 
\textbf{linearization of $\d_{\theta}$ along $s$} as the operator $\d_{\theta,s}:\Omega_S\to \Omega_S$ 
defined by
\[ \d_{\theta,s}:= \ad_{\lins \theta^{\textrm{v}}}+ \d_{\lins \Gamma_\theta}+ \ad_{\lins \Ff_\theta}.\]
This is an operator which raises the total degree by $1$.

The linearization along $s$ of the structures equations for the geometric triple 
$(\theta^{\textrm{v}},\Gamma_\theta, \Ff_\theta)$ given by Theorem \ref{Vorobjev} takes the form:

\begin{prop}
\label{linearization}
Let $\theta|_s\in\Omega_S^2$ be the element of total degree $2$ defined by:
\[ \theta|_s:= \theta^{\textrm{v}}|_{s}+ \Gamma_\theta|_{s}+ \Ff_\theta|_{s}.\]
Then:
\begin{equation}
\label{linear:structure:eqs}
\d_{\theta,s}(\theta|_{s})=0.
\end{equation}
\end{prop}

\begin{proof}
Note that equation \eqref{linear:structure:eqs} splits into the following components:
\begin{enumerate}[(i)]
\item $[\theta^{\text{\rm v}}|_{s}, \lins \theta^{\text{\rm v}}]= 0$.
\item $[\Gamma_\theta|_{s}, \lins \theta^{\text{\rm v}}]+ \d_{\lins \Gamma_\theta}(\theta^{\text{\rm v}}|_{s})= 0$.
\item $[\Gamma_\theta|_{s}, \lins \mathbb{F}]+ \d_{\lins \Gamma_\theta}(\Ff_\theta|_{s})= 0$.
\item $\d_{\lins \Gamma_\theta}(\Gamma|_{s})+ [\lins \Ff_\theta, \theta^{\text{\rm v}}|_{s}= 0$.
\end{enumerate}
The proof is straightforward: one just applies the operators $\phi_{t}^{s}$ to the structure equations given by Theorem \ref{Vorobjev} and takes the limit as $t\to 0$. Equivalently, we can expand formally in powers of $t$ 
and compare the coefficients of $t^{-1}$. 

For instance, starting from $[\theta^{\text{\rm v}}, \theta^{\text{\rm v}}]= 0$ we obtain
\[ [\frac{1}{t} \theta^{\text{\rm v}}|_{s}+ \lins \theta^{\text{\rm v}}+ \ldots, \frac{1}{t} \theta^{\text{\rm v}}|_{s}+ \lins \theta^{\text{\rm v}}+ \ldots]= 0,\]
which is equivalent to:
\[ \frac{1}{t^2} [\theta^{\text{\rm v}}|_{s}, \theta^{\text{\rm v}}|_{s}]+ \frac{1}{t} [\theta^{\text{\rm v}}|_{s}, \lins \theta^{\text{\rm v}}]+ \ldots = 0.\]
The coefficient of $t^{-1}$ gives $[\theta^{\text{\rm v}}|_{s}, \lins \theta^{\text{\rm v}}]= 0$, i.e., the first equation in the list. The other equations follow similarly. 
\end{proof}

\subsection{Linearizing Poisson structures along a symplectic leaf} 
Let $\pi$ be a horizontally non-degenerate Poisson structure on $E$ which has the zero section as a symplectic leaf. As usual, we denote by $(\pi^{\textrm{v}},\Gamma_\pi,\Ff_\pi)$ the associated geometric triple. 

The operation of restriction to the zero section only gives us the symplectic form $\omega_S$ on $S$:
\[ \pi|_S:=\pi^{\textrm{v}}|_{0}+\Gamma_\pi|_{0}+\Ff_\pi|_{0}=\omega_{S}.\]

More interesting information lies in the next order approximation, which gives us elements:
\begin{align*}
\pi^{\textrm{v}}_{\lin}&:= \d_0\pi^{\textrm{v}}\in \Gamma(\wedge^2 E\otimes E^*), \\
\sigma&:= \d_0 \Ff_\pi \in \Omega^2(S;E^*),\\
\Gamma_{\pi,\textrm{lin}}&:= \d_0\Gamma_\pi\text{ (a linear connection on $E$)}.
\end{align*}
The linearized structure equations \eqref{linear:structure:eqs}, given in Proposition \ref{linearization}, give us no information about this linearized data: those equations arose by comparing the coefficients of $t^{-1}$ in the formal expansion resulting from the structure equations. In our case, however, the next coefficients will only depend on the linearized data, and we obtain the following result of Vorobjev (\cite{Vor}):

\begin{prop}
Let $\pi$ be a horizontally non-degenerate Poisson structure on a vector bundle $p:E\to S$ which has the zero section as a symplectic leaf. The geometric triple \emph{$(\pi^{\text{v}}_{\lin}, \Gamma_{\pi,\lin}, \omega_{S}+ \sigma)$} satisfies the structure equations.
\end{prop}

Note that, in this case, $\d_{\pi, 0}$ squares to zero and $(\Omega_S,\d_{\pi, 0})$ is just the Poisson cohomology complex of $\pi$ restricted to $S$.

The proposition shows that, after restricting to an open neighborhood of the zero section
(where $\omega_{S}+ \sigma$ stays non-degenerate), the triple $(\pi^{\text{v}}_{\lin}, \Gamma_{\pi,\lin}, \omega_{S}+ \sigma)$ defines a new Poisson structure, which we will denote by $\jet^1_S\pi$
and called the \textbf{first jet approximation to $\pi$ along $S$}. It follows immediately that $\jet^1_S\pi$ has the following properties:
\begin{enumerate}[(i)]
\item The zero section $S\hookrightarrow \nu(S)$ is a symplectic leaf of $\jet^1_S\pi$;
\item The normal spaces $\nu(S)_x$, with their linear Poisson structures (which coincide with $\pi^{\text{v}}_{\lin}$) are cosymplectic submanifolds of $(\nu(S),\jet^1_S\pi)$;
\item If $S'$ is the symplectic leaf through $u\in\nu(S)$ then:
 \[ T_uS'=\Hor_u\oplus T_u \O_u,\]
where $\Hor_u$ is the horizontal space of $\Gamma_{\pi,\lin}$ and 
$\O_u\subset \nu(S)_{p(u)}$ is the coadjoint orbit through $u$.
\end{enumerate}
This leads to the following corollary, which will be useful later:

\begin{cor}
\label{cor:leaves}
If $S'$ is a symplectic leaf $j^1_S\pi$, one has:
\begin{enumerate}[(i)]
\item $S'$ projects diffeomorphically to $S$ iff it is the graph of a flat section \emph{$s\in\Gamma_\text{flat}(\nu_{S}^{0})$}, where  $\nu_{S}^{0}\subset\nu_{S}$ is the subbundle consisting of zeros of $\pi^{\text{v}}_{\lin}$. In particular the space of all such leaves coincides with $H^1_\pi(M,S)$.
\item $S'$ projects diffeomorphically to $S$ and $\omega_{S'}$ is isotopic to $\omega_S$ iff $S'$ is the graph of \emph{$s\in\Gamma_\text{flat}(\nu_{S}^{0})$} and there exists $\beta\in\Omega^1(S)$ such that:
\[ \d\beta=-\sigma|_S. \]
In particular, the space of all such leaves coincides with the image of the map $H^1_{\pi,S}(M)\to H^1_\pi(M,S)$.
\end{enumerate}
\end{cor}

\begin{proof}
From the description of $j^1_S\pi$ as the geometric triple $(\pi^{\text{v}}_{\lin},\Gamma_{\pi,\lin},\omega_{S}+\sigma)$, it is clear that the space of leaves of $\jet^1_S\pi$ that project diffeomorphically to $S$ coincides with $\Gamma_\text{flat}(\nu^0(S))$ the space of flat sections of the subbundle $\nu_{S}^{0}\subset\nu_{S}$,  consisting of the zeros of $\pi^{\text{v}}_{\lin}$. Using the Lie algebra bundle $\nu_{S}^{*}$, we can describe $\nu_{S}^{0}$ as the dual of the abelianization:
\[ \nu_{S, x}^{0}:= (\nu_{x}^{*}/ [\nu_{x}^{*}, \nu_{x}^{*}])^{*} \subset \nu_x .\]
Now the first part follows by observing that the differential $\X^0(M,S)\to X^1(M,S)$ vanishes, while the differential $X^1(M,S)\to \X^2(M,S)$ associates to a section $s\in \Gamma(\nu(S))=\X^1(M,S)$ the pair $(\d_{\pi^{\text{v}}_{\lin}} s,\nabla s)$. So $s$ is a cocycle iff it is a flat section and takes values in $\nu_{S}^{0}$. Therefore $\Gamma_\text{flat}(\nu_{S}^{0})=H^1_\pi(M,S)$.

For the second part, suppose that $s\in\Gamma_\text{flat}(\nu^0(S))$ describes a leaf $S'$ that projects diffeomorphically to $S$. From the description of $j^1_S\pi$ as a geometric triple, we see that $\omega_{S'}=\omega_S+\sigma|_s$. Therefore, by Moser's lemma, the symplectic forms are isotopic iff $[\sigma|_S]=0$. Now observe that the differential $\X^0_S(M)\to\X^1_S(M)$ associates to a function $f\in C^\infty(S)=\X^0_S(M)$ the element $(0,\d f)\in \Gamma(\nu(S))\oplus\Omega^1(S)=\X^1_S(M)$, while the differential $\X^1(M,S)\to \X^2(M,S)$ associates to a pair $(s,\beta)\in \X^1_S(M)$ the triple $(\d_{\pi^{\text{v}}_{\lin}} s,\nabla s,\d\beta+\sigma|_S)\in\X^2_S(M)$. Hence, the space of such leaves can be identified with the kernel of the map $H^1_\pi(M,S)\to H^2(S)$. But this kernel coincides with the image of the map $H^1_{\pi,S}(M)\to H^1_\pi(M,S)$.
\end{proof}

\section{Analytic framework}                                   %
\label{sec:sobolev}                                            %

In this section we consider the analytic machinery needed for the proofs. Throughout this section, $S$ is assumed to be an $n$-dimensional compact manifold.

For any vector bundle $V$ over $S$ and any integer $k\ge 0$, we denote by $\Gamma^{(k)}(V)$ the Sobolev space of order $k$ associated to $V$. We should think of it as the space of sections of $V\to S$, with the property that all distributional derivatives up to order $k$ belong to the space of $L^2$-sections of $V\to S$.  We will say that a section $s$ is of class $H^{(k)}$ if it belongs to $\Gamma^{(k)}(V)$. Note that all such spaces can be equipped with inner products $\langle \cdot , \cdot \rangle_{k}$ so that each $\Gamma^{(k)}(V)$ will be treated as a Hilbert space. Most often, we have $k>\frac{n}{2}$ so that $\Gamma^{(k)}(V)$ is a subspace of the the space $\Gamma^0(V)$ of all continuous sections of $V$ and its elements can be handled pointwise. 

One remark concerning our use of adjoints: for a bounded operator $A: \H\to \H'$ between two Hilbert spaces, including those induced by differential operators on sections of vector bundles, we will use the Hilbert space adjoint $A^*: \H'\to \H$, defined by the equation
\[ \langle Au, v \rangle_{\H'}= \langle u, A^*v\rangle _{\H} .\]

Now we discuss how the operations of restriction and linearization on a vector bundle $p:E\to S$, which we have discussed in the previous sections, can still be applied to sections which are no longer smooth, but belong instead to some Sobolev class. We emphasize that most of the times the operations themselves still make sense, since we will use Sobolev norms of orders high enough to insure that all the sections are of class $C^0$. However, we still need to assure that the result of applying these operations has the desired Sobolev class and that this process is continuous. This is achieved in the next propositions.

First, some notation: we will denote by $\Conn(E)$ (respectively, $\Conn_\text{lin}(E)$) the space of smooth (respectively, linear) connections on the vector bundle $E$. Note that $\Conn(E)$ is an affine space over $\Gamma(E; p^*\Hom(TM, E))$, so it cames equipped the $C^{p}$-topology. On the other hand, $\Conn_\text{lin}(E)$ is an affine space over $\Omega^{1}(S;\End(E))$, so we can also consider on it the $H^{(k)}$-topology, in which case we denote it by $\Conn_\text{lin}^{(k)}(E)$. 

\begin{prop}
\label{prop:analysis}
Let $E$ be a vector bundle over a compact manifold $S$ with $\dim S=n$, let $k>\frac{n}{2}$ and consider the $C^{p}$-topology on $\Omega_E$ and $\Conn(E)$. Then:
\begin{enumerate}[(i)]
\item If $p\geq k$, restriction gives continuous operations
\begin{align*}
\Omega_{E}\times  \Gamma^{(k)}(E)\to \Omega^{(k)}_S,&\quad (u,s)\mapsto u|_{s},\\
 \Conn(E)\times\Gamma^{(k)}(E)\to \Omega^{(k-1)}_S,&\quad (\Gamma,s)\mapsto \Gamma|_{s}.
\end{align*}
\item If $p\geq k+1$, linearization gives gives continuous operations
\begin{align*}
\Omega_{E}\times \Gamma^{(k)}(E)\to \Omega^{(k)}_{E,\emph{\lin}},&\quad  (u,s) \mapsto \lins u,\\
\Conn(E)\times\Gamma^{(k)}(E)\to \Conn_\emph{lin}^{(k)}(E),&\quad (\Gamma,s)\mapsto \lins \Gamma.
\end{align*}
\item If $0\leq r\leq k$, then for any $v\in\Omega^{(k)}_{E,\emph{\lin}}$ and $\Gamma\in \Conn_\text{lin}^{(k)}(E)$ one has:
\begin{align*}
\ad_{v}: \Omega^{(r)}_S& \to \Omega^{(r)}_S,\\
\d_{\Gamma}: \Omega^{(r)}_S&\to \Omega^{(r-1)}_S.
\end{align*}
Furthermore, the maps $v\mapsto\ad_{v}$ and $\Gamma\mapsto \d_{\Gamma}$ are continuous into the space of bounded linear operators.
\end{enumerate}
\end{prop}
\vskip 15 pt

\begin{rem}
In part (iii) of the proposition, we do not exclude the 
case $r\leq n/2$. In this case, the resulting elements in $\Omega_{S}^{(r)}$ are no
longer represented by continuous sections. To remove any ambiguity, one may whish to read (iii) as follows:
For each $v\in\Omega^{(k)}_{E,\text{\lin}}$ one has that $\ad_{v}$ maps $\Omega^{(k)}$ into $\Omega^{(k-1)}$ and it uniquely extends to a continuous linear map $\ad_{v}:  \Omega^{(r)}_S\to \Omega^{(r)}_S$ for all $0\leq r\leq k$
(satisfying the continuity on $v$ and $s$, as stated in the proposition). A similar comment applies to the operation $\d_{\Gamma}$.
\end{rem}

\begin{proof}
The first statement in part (i) follows from the fact that if $f(x,y)$ is a function of class $C^p$ and $s(x)$ if of class $H^{(k)}$ (recall that $p\geq k> n/2$), then $f(x, s(x))$ is again of class $H^{(k)}$ (see Lemma 9.9 in \cite{Palais}). As it is clear from the proof given in \cite{Palais}, this substitution operation is continuous with respect to the indicated topologies. For the second statement we simply note that (see (\ref{Gamma-to-s})) the expression for $\Gamma|_{s}$ is a sum involving a substitution operation and a order $1$ derivation. 

Part (ii) also holds by the same reasons as (i) and Remark \ref{purely-algebraic}.

For the first statement in part (iii), recall that for $k\geq r\geq 0$ and $k>\frac{n}{2}$, scalar multiplication
defines a continuous operation $H^k\times H^r\to H^r$ (see, e.g., Corollary 9.7 in \cite{Palais}). Since the operation $\ad_{v}$ is purely algebraic (see Remark \ref{purely-algebraic}) and it involves only fiberwise 
multiplications, the result follows. The proof of the second statement in (iii) is similar since, fixing some smooth connection $\Gamma_0$, for any other smooth connection $\Gamma$ we have $\d_{\Gamma}=\d_{\Gamma_0}+C_u$, for some $u\in \Omega^{1, (k)}(S;\textrm{End}(E))$, where the operation $C_{u}$ (the ``composition with $u$'') is again purely algebraic and behaves like a fiberwise multiplication. 
\end{proof}

\begin{cor} 
Let $p\geq k+ 1> \frac{n}{2}+ 1$. Consider the $C^p$-topology on the space $\X^2_\text{HN}(E)$ of (smooth) horizontally nondegenerate bivectors $\theta$ on $E$. Then:
\begin{enumerate}[(i)]
\item For $\theta\in\X^2_\text{HN}(E)$ and $s\in \Gamma^{(k)}(E)$, we have
\[ \theta|_{s}= (\theta^{\text{v}}|_{s},\Gamma_\theta|_{s},\Ff_\theta|_{s})\in \Omega^{(k-1)}_S\]
and the operation $(\theta, s)\mapsto \theta|_s$ is continuous.
\item For $\theta\in\X^2_\text{HN}(E)$ and $s\in \Gamma^{(k)}(E)$ and all $0\leq r\leq k$, the operator  
\[ \d_{\theta, s}= \ad_{\lins \theta^{\text{v}}}+ \d_{\lins \Gamma_\theta}+ \ad_{\lins \Ff_\theta} : \Omega^{(r)}_S\to \Omega^{(r-1)}_S\]
is continuous and so is the operation $(\theta, s)\mapsto \d_{\theta, s}$.
\item For $\theta\in\X^2_\text{HN}(E)$ and $s\in \Gamma^{(k)}(E)$ the linearized structure equations hold:
\[ \d_{\theta, s}(\theta|_s)= 0.\]
\end{enumerate}
\end{cor}

\begin{proof} 
To prove (i) one applies Proposition \ref{prop:analysis} (i) to the components of the geometric triple induced by $\theta$. To prove (ii), one applies first Proposition \ref{prop:analysis} (ii) to the components of the same geometric triple and then one applies Proposition \ref{prop:analysis} (iii) with $v=\lins \theta^{\text{v}}$, with $v=\lins \Ff_\theta$ and with $\Gamma=\lins \Gamma_\theta$. The linearized structure equations follows from Proposition \ref{linearization}.
\end{proof}

From the discussion above, it follows that for $k> \frac{n}{2}$ and $s\in \Gamma^{(k)}(E)$, we have a sequence 
\begin{equation}
\label{completed-nonbar-seq} 
\xymatrix{
\Omega^{1, (k)}_{S} \ar[r]^{\d_{\theta, s}}& \Omega^{2,(k-1)}_{S} \ar[r]^{\d_{\theta, s}}& \Omega^{3,(k-2)}_{S}}
\end{equation} 
and a similar sequence for the relative complex: 
\begin{equation}
\label{completed-seq} 
\xymatrix{
\overline{\Omega}^{1, (k)}_{S} \ar[r]^{\overline{\d}_{\theta, s}}& \overline{\Omega}^{2,(k-1)}_{S} \ar[r]^{\overline{\d}_{\theta, s}}& \overline{\Omega}^{3,(k-2)}_{S}}
\end{equation} 
(for the notations, see the end of Section \ref{sec:hor:non-degenerate}). The key technical lemma, and the main reason for 
our choice of Sobolev norms, is the following result.

\begin{prop}
\label{prop:elliptic} 
Let $\pi$ be a horizontally non-degenerate Poisson structure on $E$ and we put $\d_{\pi}= \d_{\pi, 0}$.
Then $(\Omega^{\bullet}_{S}, \d_{\pi})$ and $(\overline{\Omega}^{\bullet}_{S}, \overline{\d}_{\pi})$ are elliptic complexes. 

If $H^{2}_{\pi,S}(M)=0$, we have:
\begin{enumerate}[(i)]
\item The sequence (\ref{completed-nonbar-seq}) (with $\theta=\pi$ and $s=0$) is exact. 
\item For each $p$ and each $r$, $\d_{\pi}(\Omega_{S}^{p,(r)})$ is closed in $\Omega_{S}^{p+1,(r-1)}$.
\item The ``Laplacian'' $\lap_{\pi}:= \d_\pi \d_\pi^{*}+ \d_\pi^{*}\d_\pi: \Omega^{2,(k-1)}_{S}\to \Omega^{2,(k-1)}_{S}$ is an isomorphism.
\end{enumerate}

If $H^{2}_{\pi}(M, S)= 0$, similar statements hold for the complex $(\overline{\Omega}^{\bullet}_{S},\overline{\d}_{\pi})$. 
\end{prop}

\begin{proof}
First of all $(\Omega^{\bullet}_{S}, \d_{\pi})$ is the de Rham complex of a transitive Lie algebroid, namely of $T^*M|_{S}$. It is well-known (exactly by the same arguments as in the case of ordinary de Rham complexes) that the de Rham complexes of transitive Lie algebroids are elliptic. This, together with the ellipticity of the de Rham complex of $S$, also implies the ellipticity of the quotient $\overline{\Omega}^{\bullet}_{S}= \Omega_{S}^{\bullet}/\Omega^{\bullet}(S)$: the symbol complex of $(\overline{\Omega}^{\bullet}_{S}, \overline{\d}_{\pi})$ at non-zero $\xi\in T^*M$ is the quotient of the symbol complex of $\Omega^{\bullet}_{S}$ at $\xi$ modulo the symbol complex of $\Omega^{\bullet}(S)$, hence acyclic (being a quotient of two acyclic complexes). 

For the proof of the remaining statements, we invoke some general facts about elliptic complexes (see, e.g., \cite{Gilk}), as we now explain. So let us consider an elliptic complex of order $m$ over a compact manifold $S$:
\[
\xymatrix{\ldots\ar[r]& \Gamma(E^{p-1})\ar[r]^{D}& \Gamma(E^{p}) \ar[r]^{D}& \Gamma(E^{p+1}) \ar[r]& \ldots }  
\]

Part (i), follows from the following fact: if the complex is exact at level $p$, then the same is true for the completed sequence
\[ 
\xymatrix{\ldots\ar[r]& \Gamma^{(r+m)}(E^{p-1}) \ar[r]^{D}&\Gamma^{(r)}(E^{p})\ar[r]^{D}& \Gamma^{(r-m)}(E^{p+1}) \ar[r]& \ldots }\]
To see this, recall that there exists a paramatrix $P$ for $D$, i.e., pseudo-differential operators of order $-m$ satisfying $PD+DP=\text{Id}-K$, for some smoothing operator $K$. Therefore, if 
$u\in \Gamma^{(r)}(E^{p})$ satisfies $D(u)=0$, it follows that
\[ u=K(u)+D(P(u)).\]
Since $K(u)$ is smooth and is killed by $D$, we have $K(u)= D(v)$ for some smooth $v$. On the other hand,
$P(u)\in  \Gamma^{(r-m)}(E^{p+1})$, because $P$ is of order $-m$. Hence $w=v+P(u)$ belongs to $\Gamma^{(r-m)}(E^{p+1})$ and $u=D(w)$, as desired.

Part (ii) follows from another standard fact for elliptic complexes over compact manifolds: $D(\Gamma^{(m)}(E^p))$ is closed in $\Gamma^{(0)}(E^{p+1})$. Using this, let us check that $D(\Gamma^{(r)}(E^p))$ is closed in $\Gamma^{(r-m)}(E^{p+1})$. So assume that $u_n\in \Gamma^{(r)}(E^p)$ and that $D(u_n)$ converges to $u$ in $\Gamma^{(r-m)}(E^{p+1})$: we want to show that $u=D(v)$, for some $v\in\Gamma^{(r)}(E^p)$. Using a paramatrix as above, we write $u= K(u)+ DP(u)$. If we set $v_n:= K(u_n)$, we find that $D(v_n)= KD(u_n)$. Since $K$ is smoothing, the $v_n$ will be smooth and $D(v_n)$ converges to $K(u)$ in $\Gamma^{(0)}(E^{p+1})$. It follows that
$K(u)=D(v')$ for some $v'\in\Gamma(E^p)$, so that $u=D(v)$ where $v=v'+P(u)\in \Gamma^{(r)}(E^p)$.

To prove part (iii), note that (i) and (ii) show that we are in the following situation: we have a sequence of Hilbert spaces and bounded operators
\[ U\stackrel{A}{\to} V \stackrel{B}{\to} W\]
which is exact in the middle and with the property that the range of $B$ is closed in $W$. We claim that, under such circumstances, $\lap:= AA^*+ B^*B: V\to V$ is an isomorphism. To see this, decompose
\[ U= U_0\oplus U_{0}^{\perp},\quad V= V_0+ V_{0}^{\perp},\quad W= W_0\oplus W_{0}^{\perp},\]
where $U_0=\Ker(A)^\perp$, $V_0=\im(A)=\Ker(B)$ and $W_0=\im(B)$. Note that our hypotheses guarantee that these are all Hilbert space decompositions. Moreover, we have that
\[ A_{0}:= A|_{U_0}: U_{0}\rmap V_0,\quad B_{0}:= B|_{V_{0}^{\perp}}: V_{0}^{\perp}\rmap W_0 \]
are isomorphisms (continuous bijections between Hilbert spaces). With respect to the above decomposition,
$\lap= \textrm{diag}( A_{0}A_{0}^*, B_{0}^*B_{0})$, hence it is an isomorphism.
\end{proof}

\section{Proofs: the Poisson case}                             %
\label{sec:proof:Poisson}                                      %

In this section we prove first Theorem \ref{theorem11}, concerning stability. Then we indicate the changes necessary to prove Theorem \ref{theorem22} about strong stability. The last paragraph contains the proof of the necessity conditions given in Theorem \ref{thm:necessity:Poisson}.

\subsection{Proof of Theorem \ref{theorem11}} 
We now place ourselves under the assumptions of Theorem \ref{theorem11} and we set $k=\kappa-1$. For the proof we will use the reduced complex $\overline{\Omega}_{S}$ and the operators $\overline{\d}_{\theta, s}$ acting on it. Recall that
\[ \overline{\Omega}_{S}^{\,l}=\sum_{\substack{p+q=l\\ q\geq 1}} \Omega^{p}(S; \wedge^q E), \ \overline{\d}_{\theta, s}= \ad_{\lins \theta^{\text{\rm v}}}+ \d_{\lins \Gamma}+ \ad_{\lins \mathbb{F}}.\]

Given a section $s\in\Gamma^{(k)}(E)$, we consider the element 
\[ c(\theta,s): = (\theta^{\text{\rm v}}|_{s},\Gamma_\theta|_{s}) \in \overline{\Omega}_{S}^{2,(k-1)}.\]
We recall that our aim is to look for a sections $s$ such that this expression is zero. Note that such an $s$ is necessarily smooth, hence it will define a symplectic leaf of $\theta$ (Proposition \ref{criteria-leaf}).
We introduce the functional $\Phi_{\theta}: \Gamma^{(k)}(E)\to \mathbb{R}$:
\[ 
\Phi_{\theta}(s):= ||c(\theta, s)||_{(k-1)}^{2}= ||\theta^{\text{\rm v}}|_{s}||_{(k-1)}^{2}+||\Gamma|_{s}||_{(k-1)}^{2},\]
where, as before, the subscript $(k-1)$ indicates the use of the Sobolev inner product of class $k-1$.

\begin{lem} 
\label{lem:differential}
$\Phi_{\theta}$ is smooth and for each $s,w\in \Gamma^{(k)}(E)=\overline{\Omega}_{S}^{1,(k)}$ we have
\[ (\d\Phi_{\theta})|_s(w)= -\langle \overline{\d}_{\theta,s}(w), c(\theta, s)\rangle_{(k-1)} .\]
\end{lem}

\begin{proof}
For $w\in \overline{\Omega}_{S}^{1,(k)}$ we compute:
\begin{align*} 
(\d\Phi_{\theta})|_s(w)
&=\left.\frac{\d}{\d t}\right|_{t=0} \langle c(\theta,s+tw),c(\theta,s+tw)\rangle_{(k-1)}\\
&=2 \langle \left.\frac{\d}{\d t}\right|_{t=0} c(\theta,s+tw),c(\theta,s)\rangle_{(k-1)}.
\end{align*}
Hence it remains to check that the operation 
\[ \overline{\Omega}_{S}^{1,(k)}\ni s\longmapsto c(\theta, s)\in \overline{\Omega}_{S}^{2,(k-1)}\]
is smooth, with first derivative precisely $-\overline{\d}_{\theta,s}$:
\begin{equation}
\label{variation} 
-\left.\frac{\d}{\d t}\right|_{t=0} c(\theta,s+tw)= \overline{\d}_{\theta,s}(w),\quad (w\in \overline{\Omega}_{S}^{1,(k)})
\end{equation}
We can check that this holds, by checking each component, so we need to show that:
\[ 
-\left.\frac{\d}{\d t}\right|_{t=0}  \Gamma|_{s+ tw}= \d_{\lins \Gamma}(w),\quad 
-\left.\frac{\d}{\d t}\right|_{t=0} \theta^{\text{\rm v}}|_{s+tw}= \ad_{\lins \theta^{\text{\rm v}}}(w).
\]
We check the first identity by computing in local coordinates (the second one is checked in a similar fashion and so is left for the interested reader):
\begin{align*}
-\left.\frac{\d}{\d t}\right|_{t=0}  \left(\Gamma|_{s+ tw}\right)_i^a
&=\left.\frac{\d}{\d t}\right|_{t=0} \left(\frac{\partial (s+tw)^a}{\partial x_i}-\Gamma_{i}^{a}(x,s(x)+ tw(x))\right)\\
&= \frac{\partial w^a}{\partial x^i}- \frac{\partial \Gamma_{i}^{a}}{\partial y_b}(x, s(x)) w^b(x)
\end{align*}
which are precisely the local expressions for the covariant derivative of $w$ with respect to the linear connection $\lins \Gamma$. 
\end{proof}

When $\theta=\pi$, the zero section $s=0$ is a critical point of $\Phi_{\pi}$. The next result gives its Hessian.

\begin{lem} 
\label{lem:hessian}
For $w_1,w_2\in\Gamma^{(k)}(E)=\overline{\Omega}_{S}^{1,(k)}$, we have:
\[ (\d^2\Phi_{\pi})|_0(w_1, w_2)= \langle \overline{\d}_{\pi, 0}(w_1), \overline{\d}_{\pi, 0}(w_2)\rangle_{(k-1)}.\]
\end{lem}

\begin{proof}
We compute, using Lemma \ref{lem:differential}:
\begin{align*} 
(\d^2\Phi_{\pi})|_0(w_1, w_2)
&= \left.\frac{\d}{\d t}\right|_{t=0}(\d\Phi_{\pi})|_{tw_1}(w_2)\\
&= -\left.\frac{\d}{\d t}\right|_{t=0}\langle \overline{\d}_{\pi, tw_1}(w_2), c(\pi, tw_1) \rangle_{(k-1)}.
\end{align*}
Since $c(\pi,0)=0$, we conclude that:
\[ 
(\d^2\Phi_{\pi})|_0(w_1, w_2)=-\langle \overline{\d}_{\pi,0}(w_2), \left.\frac{\d}{\d t}\right|_{t=0} c(\pi, tw_1)\rangle_{(k-1)}. \]
Applying (\ref{variation}), we obtain the desired formula.
\end{proof}

The previous results show that we have the Hilbert space decomposition:
\[ \overline{\Omega}_{S}^{1,(k)}=\Ker \overline{\d}_{\pi,0}\oplus (\Ker \overline{\d}_{\pi,0})^\perp. \] 
For the sequel, we set $K:=\Ker \overline{\d}_{\pi,0}$ and $W:=K^\perp$. We have:

\begin{lem} 
The origin $0\in W$ is a strongly non-degenerate critical point of $\Phi_{\pi}|_{W}: W\to \mathbb{R}$.
\end{lem}

\begin{proof} 
This follows from the computation of the Hessian given in Lemma \ref{lem:hessian} and the fact that 
$\overline{\d}_{\pi, 0}(\overline{\Omega}_{S}^{1,(k)})$ is closed in $\overline{\Omega}_{S}^{2,(k-1)}$ (see Proposition \ref{prop:elliptic} (ii)).
\end{proof}

We have the following simple application of the inverse function theorem.

\begin{prop} 
Let $\H$ be a Hilbert space, let $\Lambda$ be a topological space and let $F_{\lambda}:\H\to\Rr$ be a family of $C^2$-maps depending continuously on a parameter $\lambda\in\Lambda$. If $0\in \mathcal{H}$ is a strongly non-degenerate critical point of $F_{\lambda_0}$, then there is a continuous map $U\ni\lambda\mapsto x_{\lambda}\in\H$, defined in a neighborhood $U\subset\Lambda$ of $\lambda_0$, such that each $x_{\lambda}$ is a critical point of $F_{\lambda}$.
\end{prop}

\begin{proof} 
Let $G(\cdot,\lambda)$ be the gradient of $F_{\lambda}$:
\[ G:\H\times\Lambda\to\H^*\simeq\H,\quad G(x,\lambda):=\d_x F_\lambda. \]
By our assumptions on $F_\lambda$, the partial derivative $D_1G(x,\lambda)$ is continuous in both $x$ and $\lambda$, while $D_1G(0,\lambda_0):\H\to\H$ is an isomorphism. We can then apply the version of the inverse function theorem with parameters (see, e.g., \cite[Theorem 2]{Nij}) to conclude that the map $(x,\lambda)\mapsto (G(x,\lambda),\lambda)$ has a continuous inverse $(y,\lambda)\mapsto(G^{-1}(y,\lambda),\lambda)$, defined in some neighborhood of $(0,\lambda_0)$. If we now set
$x_\lambda:=G^{-1}(0,\lambda)$, the result follows.
\end{proof} 

\begin{rem}
Assume that the parameter space is a product $\Lambda=\Lambda'\times \H'$, where $\Lambda'$ is some topological space and $\H'$ is a Hilbert space. If $F_\lambda=F_{(\theta,s)}$ dependends continuously on $\theta\in\Lambda'$ and smoothly on $s\in H'$, then the family of critical points $x_\lambda=x_{(\theta,s)}$ will depend continuously on $\theta$ and smoothly on $s$. The proof is the same, except that one now uses the version of the \emph{implicit} function theorem with parameters.
\end{rem}

We can now proceed to complete the proof of Theorem \ref{theorem11}. The family of functionals indexed by $(\theta,s)$, with $s\in K$, defined by:
\[ W\ni w\mapsto \Phi_{\theta}(s+w),\]
has a strongly non-degenerate critical point at the origin when $\theta=\pi$. Hence, there exists a neighborhood of $\pi$ and a neighborhood of $0$ in $K$, where the map $(\theta,s)\mapsto w_{\theta,s}\in W$ is defined, depends continuously on $\theta$ and smoothly on $s$, and
\[ \tilde{s}_{\theta, s}:= s+ w_{\theta, s}\] 
is a critical point of $\Phi_{\theta}|_{s+W}$. Note that $\tilde{s}_{\theta,s}$ can be assumed to be as small as we wish (in $\overline{\Omega}_{S}^{1,(k)}$), by considering small enough neighborhoods.

We claim that $\tilde{s}_{\theta,s}$ is a symplectic leaf of $\theta$, i.e., that it is a smooth section satisfying $c(\theta,\tilde{s}_{\theta,s})=0$ (see Proposition \ref{criteria-leaf}). To prove this, we start by remarking
that $\tilde{s}_{\theta, s}$ being a critical point means that (see Lemma \ref{lem:differential}):
\[ (\d\Phi_{\theta})|_{\tilde{s}_{\theta, s}} (w)=\langle \overline{\d}_{\theta,\tilde{s}_{\theta,s}}(w),c(\theta,\tilde{s}_{\theta,s})\rangle_{(k-1)}=0,\quad \forall w\in W.\]
This shows that:
\[ 
\overline{\d}_{\theta,\tilde{s}_{\theta,s}}^{*}(c(\theta,\tilde{s}_{\theta,s}))\in W^{\perp}= K
\quad \Longrightarrow\quad 
\overline{\d}_{\pi,0} \overline{\d}_{\theta, \tilde{s}_{\theta,s}}^{*}(c(\theta,\tilde{s}_{\theta,s}))= 0.\]
We recall that we also have $\overline{\d}_{\theta,\tilde{s}_{\theta,s}}(c(\theta,\tilde{s}_{\theta,s}))=0$. So, if we consider the operator 
\[ 
\overline{\lap}_{\theta,\tilde{s}_{\theta,s}}:=
\overline{\d}_{\pi,0}\overline{\d}_{\theta,\tilde{s}_{\theta,s}}^{*}+ \overline{\d}_{\pi,0}^{*}\overline{\d}_{\theta, \tilde{s}_{\theta,s}}: \overline{\Omega}_{S}^{2,(k-1)}\to \overline{\Omega}_{S}^{2,(k-1)}\]
we find that
\[ \overline{\lap}_{\theta, \tilde{s}_{\theta,s}}(c(\theta, \tilde{s}_{\theta,s}))= 0.\]
By Proposition \ref{prop:elliptic} (iii), we know that $\overline{\lap}_{\pi}=\overline{\lap}_{\pi,0}$ is an isomorphism. Hence so are $\overline{\lap}_{\theta, \tilde{s}_{\theta,s}}$, for $s$ small enough and $\theta$ close enough to $\pi$. It follows that 
$c(\theta,\tilde{s}_{\theta,s})= 0$. Note that, at this point we only know that $\tilde{s}_{\theta,s}$ is of type $H^{(k)}$. However, the equation $c(\theta, \tilde{s}_{\theta,s})= 0$ is elliptic on $\tilde{s}_{\theta,s}$, so $\tilde{s}_{\theta,s}$ must be smooth. 

In conclusion, we have proved the existence of a smooth section $\tilde{s}_{\theta,s}$, depending continuously on $\theta$ and smoothly in $s\in K$, such that
\[ \tilde{s}_{\theta,s}\in s+ W\]
and $\Graph(\tilde{s}_{\theta,s})$ is a symplectic leaf of $\theta$. Now observe that elements in $K$, being flat sections, are necessarily smooth, so $K=\Ker \overline{\d}_\pi=\Gamma_{\text{flat}}(\nu^0(S))=H^1_\pi(M,S)$ (Proposition \ref{criteria-leaf} and Corollary \ref{cor:leaves} (i)) and this concludes the proof of Theorem \ref{theorem11}.

\subsection{Proof of Theorem \ref{theorem22}}
We now indicate the changes that are necessary to prove Theorem \ref{theorem22}.

The proof follows exactly the same pattern. As before, we set $k=\kappa-1$. We now consider the complex $\Omega^\bullet_S$ and the element $c(\theta,s,\beta)\in \Omega^{2,(k-1)}_S$ which, for each horizontally non-degenerate Poisson structure $\theta$, with associated geometric triple $(\theta^{\text{\rm v}}|_{s},\Gamma_\theta,\Ff_\theta)$, and each pair $(\beta,s)\in\Gamma^{(k)}(E)\oplus\Omega^{1,(k)}(S)=\Omega^{1,(k)}_S$, is defined by:
\[ 
c(\theta,s,\beta): = (\Gamma_\theta|_{s},\theta^{\text{\rm v}}|_{s},\Ff_\theta|_s-\omega_S+\d\beta) .
\]
A small enough pair $(s,\beta)$, with $s$ and $\d\beta$ smooth, such that $c(\theta,s,\beta)=0$ corresponds to a symplectic leaf $\Graph(s)$ which projects diffeomorphically to $S$ and which has symplectic form isotopic to $\omega_S$. This is just an extension of Proposition \ref{criteria-leaf} (see, e.g., the proof of Corollary \ref{cor:leaves} (ii)). Note that the equation $\d_{\theta,s}(\theta|_{s})= 0$ immediately implies:
\[ \d_{\theta,s}c(\theta,s,\beta)=0. \]

One now studies the functional $\Phi_\theta:\Omega^{1,(k)}_S\to\Rr$ defined by:
\[ \Phi_\theta(s,\beta):=||c(\theta,s,\beta)||^2_{(k-1)}, \]
which is smooth, with differential
\[ (\d\Phi_\theta)|_{(s,\beta)}(w)=-\langle \d_{\theta,s}w,c(\theta,s,\beta)\rangle_{(k-1)}. \]
When $\theta=\pi$ and $(s,\beta)=0$, this functional has a critical point with Hessian given by:
\[ (\d^2\Phi_\pi)|_{0}(w_1,w_2)=\langle \d_{\pi,0}w_1,\d_{\pi,0}w_2 \rangle_{(k-1)}. \]
Hence, set $K:=\Ker \d_{\pi,0}$ and $W:=(\Ker \d_{\pi,0})^\perp$. It follows that for $(s,\beta)\in K$
the functional defined by
\[ W\ni w\mapsto \Phi_\theta((s,\beta)+w),\]
has a strongly non-degenerate critical point at the origin when $\theta=\pi$. 

Hence, there exists a neighborhood of $\pi$ and a neighborhood of the origin in $K$, where the map $(\theta,s,\beta)\mapsto w_{\theta,s,\beta}\in W$ is defined, and 
\[ (\tilde{s}_{\theta, s,\beta},\tilde{\beta}_{\theta, s,\beta}):= (s,\beta)+w_{\theta,s,\beta}\] 
is a critical point of $\Phi_{\theta}|_{(s,\beta)+W}$. Note that $w_{\theta,s,\beta}$ depends continuously on $\theta$ and smoothly on $(s,\beta)$.

One needs to show that:
\[ c(\theta,\tilde{s},\tilde{\beta})=0.\]
This follows by exactly the same method as in the proof of Theorem \ref{theorem11}: One applies the ``Laplacian'' type operator $\lap_{\theta,s}:\Omega^{2,(k-1)}_S\to \Omega^{2,(k-1)}_S$ defined by:
\[ \lap_{\theta,s}:=\d_{\theta,s}\d_{\theta,s}^*+\d_{\theta,s}^*\d_{\theta,s}, \]
to this element, and checks that:
\[ \lap_{\theta,s}(c(\theta,\tilde{s},\tilde{\beta}))=0. \]
By Proposition \ref{prop:elliptic}, when $\theta=\pi$ and $s=0$, the operator $\lap_{\pi,0}$ is an isomorphism. Hence, for $\theta$ and $s$ small enough we have that $\lap_{\theta,s}$ is also an isomorphism, so we must have $c(\theta,\tilde{s},\tilde{\beta})=0$.

To complete the proof, note that elements $(s,\beta)\in K$ must have $s$ smooth (since $s$ is a flat section) and that pairs $(\tilde{s}_{\theta, s,\beta},\tilde{\beta}_{\theta, s,\beta})$ and $(\tilde{s}_{\theta, s',\beta'},\tilde{\beta}_{\theta, s',\beta'})$ define the same symplectic leaf of $\theta$ precisely when $s=s'$ and $-\d\beta=\sigma|_s=\sigma|_{s'}=-\d\beta'$ (by Corollary \ref{cor:leaves} (ii)). Since $\sigma|_s$ is smooth, we can always choose a smooth 1-form $\beta$ with $-\d\beta=\sigma|_S$. Hence, the space of symplectic leaves is parametrized by:
\[ \{s\in\Gamma_{\text{flat}}(\nu^0(S)):\sigma|_s\in\d(\Omega^1(S))\} \]
and this is precisely the image of the map $H^1_\pi(M,S)\to H^1_{\pi,S}(M)$.

\subsection{Proof of Theorem \ref{thm:necessity:Poisson}}
To complete the discussion of the Poisson case, we are missing the proof of the necessary condition for stability given in Theorem \ref{thm:necessity:Poisson}.

We place ourselves under the assumptions of Theorem \ref{thm:necessity:Poisson}. Thus we can assume that we have a vector bundle $p:E\to S$ with a Poisson structure $\pi\in\X^2(E)$ for which the zero section is a compact symplectic leaf $(S,\omega_S)$. The Poisson structure $\pi$ is horizontally non-degenerate and the corresponding triple $(\pi^\text{v},\Gamma_\pi,\Ff_\pi)$ is formed by a \emph{linear} vertical Poisson structure, a \emph{linear} connection, and an \emph{affine} 2-form: $\Ff_\pi=\omega_S+\sigma$, where $\sigma\in\Omega^{0,2}_\text{lin}(E)$.

Let $c\in\Omega^2_S(E)$ and split it into its components:
\[ c=c_{2,0}+c_{1,1}+c_{0,2}\in\Omega^{2,0}_S\oplus\Omega^{1,1}_S\oplus\Omega^{0,2}_S. \]
To this element we associate a family of horizontally non-degenerate bivector fields $\pi_t\in\X^2(E)$, with $t\in\Rr$, defined by the geometric triple:
\begin{align*}
(\pi_t)^\text{v}&:=\pi^\text{v}+t\,c_{2,0},\\
\Gamma_{\pi_t}&:=\Gamma_\pi+t\,c_{1,1},\\
\Ff_{\pi_t}&:=\Ff_\pi+t\,c_{0,2}=\omega_S+\sigma+t\,c_{0,2}.
\end{align*}
The properties of $[\cdot,\cdot]$, in particular the fact that it vanishes on $\Omega_S$, show that:
\[ [\pi_t,\pi_t]=0\quad\Longleftrightarrow\quad \d_{\pi,0}c=0.\]
Hence, if $c$ is a cocycle, then $\pi_t$ is a family of Poisson structures with $\pi_0=\pi$. 

In order to prove part (i) of the theorem, we assume that $S$ is a stable leaf. We need to show that the map $\Phi:H^2_{\pi,S}(E)\to H^2_\pi(E,S)$ vanishes, i.e., that given a cocycle $c\in\Omega^2_S(E)$, then the cocycle 
\[ 
\phi(c):=c_{2,0}+c_{1,1}\in \overline{\Omega}_{S}^2
\] 
is exact. For this, observe that for $t$ small enough, the associated Poisson structure $\pi_t$ is close to $\pi$ and so must have a leaf $S'$ close to $S$, which projects diffeomorphically under $p:E\to S$. By Proposition \ref{criteria-leaf}, the leaf $S'$ is of the form $\Graph(s)$ for a section $s\in\Gamma(E)$ satisfying:
\[ (\pi_t)^\text{v}|_s=0,\quad \Gamma_{\pi_t}|_s=0. \]
The expressions of $(\pi_t)^\text{v}$ and $\Gamma_{\pi_t}$ given above, show that this can be written:
\[ \d^{1,0}_\pi s=-t\,c_{2,0},\quad \d^{0,1}_\pi s=-t\,c_{1,1}.\]
This just means that $\overline{\d}_{\pi,0}\left(\frac{1}{t}\,s\right)=-\phi(c)$, so $\phi(c)$ is exact.

Now we prove part (ii) of the theorem. We assume that $S$ is a strongly stable leaf and we show that $H^2_{\pi,S}(E)=0$. For this, we choose a cocycle $c\in\Omega^2_S(E)$ and we form the associated family of Poisson structures $\pi_t$. Since $S$ is strongly stable, for $t$ small enough, $\pi_t$ must have a leaf $S'$ close to $S$, which projects diffeomorphically under $p:E\to S$, and with symplectic form isotopic to $\omega_S$. This means that $S'$ is described by a section $s$ such that:
\[ (\pi_t)^{\text{v}}|_s=0,\quad \Gamma_{\pi_t}|_s=0,\quad \Ff_{\pi_t}|_s-\omega_S\in\d(\Omega^1(S)) . \]
The first two conditions, as we saw above, are equivalent to: 
\[ \d^{1,0}_\pi s=-t\,c_{2,0}, \quad\d^{0,1}_\pi s=-t\,c_{1,1}.\]
The expression for $\Ff_t$ shows that the last condition amounts to:
\[ \sigma|_s+t\,c_{0,2}=\d\beta \quad\Longleftrightarrow\quad \d^{-1,2}_\pi s-\d^{0,1}_\pi\beta=-tc_{0,2}, \]
for some $\beta\in\Omega^1(S)$. We conclude that $\d_\pi \left(\frac{1}{t}(s-\beta)\right)=-c$, so $c$ is exact.

\section{Proofs: the Lie algebroid case}                             %
\label{sec:proof:algebroid}                                          %

We now turn to the proof of Theorem \ref{theorem33}. Since this proof is again very similar to that of Theorem \ref{theorem11}, we will just indicate the setup and the main steps. Before that, we recall the construction of the deformation complex of a Lie algebroid (see \cite{CrMo}) and its relationship with Poisson geometry.

Given a vector bundle $A$ over $M$, one defines $C^{q}_{\textrm{def}}(A)$ as the space of $q$-multi-derivations on $A$, i.e., skew-symmetric multilinear maps:
\[
\underbrace{\Gamma(A)\otimes\cdots\otimes\Gamma(A)}_{q\ \text{times}}
\to \Gamma(A)
\]
which is a derivation in each entry. This means that there exists a skew-symmetric bundle map
\[ \sigma_D:\wedge^{q-1} A\to TM,\]
called the \emph{symbol}, such that
\[
D(\alpha_1,\dots, f\alpha_q)=fD(\alpha_1,\dots,\alpha_q)+\sigma_{D}(\alpha_1,\dots,\alpha_{q-1})(f)\alpha_q,
\]
for any function $f\in C^\infty(M)$ and sections $\alpha_1,\dots,\alpha_q\in\Gamma(A)$. From this definition, it is clear that a Lie algebroid structure on $A$ can be interpreted as an element $\ell \in C^{2}_{\textrm{def}}(A)$, namely the Lie bracket of the structure, with symbol the anchor. The condition that $\ell$ must satisfy makes use of a certain graded Lie bracket $[\cdot, \cdot]$ on $C^\bullet_{\textrm{def}}(A)$, which will be recalled below. Using this bracket, Lie algebroid structures on $A$ are in 1-1 correspondence with elements $\ell\in C^{2}_{\textrm{def}}(A)$ which satisfy $[\ell,\ell]=0$.

The space $C^\bullet_{\textrm{def}}(A)$ can also be interpreted as the space of linear multivector fields on the manifold $A^*$, which provides a close relationship with Poisson geometry. By a linear $q$-vector field on $A^*$ we mean a $q$-vector field $X\in\X^q(A^*)$ which locally looks like:
\[
a^j_J(x) \xi_j \partial_{\xi_{J}}+\ b^i_{J'}(x) \partial_{\partial x^{i}}\wedge \partial_{\xi_{J'}},
\]
where $x=(x^1,\ldots,x^m)$ are coordinates in $M$, $(\xi_1,\ldots,\xi_r)$ are fiberwise coordinates in $A^*$ coming
from a local frame of $A$, $a^j_J$ and $b^i_{J'}$ are smooth functions, while $J$ and $J'$ are multi-indices of length $q$ and $q-1$, respectively. We will denote this space by $\X_{\textrm{lin}}^{q}(A^*)$.

A coordinate free description of the isomorphism
\[ \X^{q}_{\textrm{lin}}(A)\cong C^{q}_{\textrm{def}}(A), \ X\mapsto D_X\]
associates to $X$ the derivation $D_X$ determined by:
\[ e_{D_X(\alpha_1, \ldots , \alpha_q)}= \langle \d e_{\alpha_1}\wedge \ldots \wedge \d e_{\alpha_q}, X\rangle,\]
where for $\alpha\in \Gamma(A)$ we denote by $e_{\alpha}\in C^{\infty}(A^*)$ the fiberwise linear
function on $A$ induced by $\alpha$. The symbol of $D_X$ is given by:
\[ \sigma_X(\alpha_1, \ldots , \alpha_{q-1})(f)= \langle \d e_{\alpha_1}\wedge \ldots \wedge \d e_{\alpha_{q-1}} \wedge \d\tilde{f},X\rangle\]
for all $f\in C^{\infty}(M)$, where $\tilde{f}= f\circ p$. Finally, the bracket on $C^{\bullet}_{\textrm{def}}(A)$ mentioned above corresponds to the restriction of the Schouten bracket to $\X^{\bullet}_{\textrm{lin}}(A^*)$. For our purposes, this can be taken as the definition of the bracket on $C^{\bullet}_{\textrm{def}}(A)$.


We now turn to the proof of Theorem \ref{theorem33}. We denote by $\ell_0\in C^{2}_{\textrm{def}}(A)$ the
original Lie algebroid structure on $A$. In order to emphasize the analogy with the previous proof, we will denote by $S$, instead of $L$, a compact leaf of $\ell_0$. As in the Poisson case, after choosing a tubular neighborhood of $S$, we may assume that $M$ is a vector bundle $p:E\to S$. We may also assume that, as a vector bundle, $A=p^*A_S$ where $A_S= A|_{S}$.


The space that plays the role analogue of the space $\Omega_{E}$ for the Poisson case, is the space
\[ \mathcal{C}_{E}^{\bullet}:= C^{\bullet}_{\textrm{def}}(A),\]
endowed itself with the bracket coming from the Schouten bracket on multivector fields on $A^*$. On the other hand, the analogue of $\Omega_{S}$ is the space
\[ \mathcal{C}_{S}^{\bullet}:= \Omega^{\bullet-1}(A_S; E).\]
All the spaces that we consider will be subspaces of $\mathcal{C}_E$. For instance, $\mathcal{C}_{S}^{q}$
sits naturally inside $\mathcal{C}_{E}^{q}$: any $c\in \Omega^{q-1}(A_S; E)$ induces a derivation $D\in C^{q}_{\textrm{def}}(A)$ which is zero on sections constant on the fibers and has symbol $\sigma_D=c$ (use the natural inclusion of $p^*E$ into $TE$).

We will also use local coordinates to express elements in $\mathcal{C}_E$. For that, we choose local coordinates  $(x^i)$ in $S$ and fiberwise coordinates $(y^j)$ on $E$, which together give coordinates on the total space $E$. We also choose fiberwise coordinates $(\xi_k)$ on $A=p^*A_S$ coming from fiberwise coordinates on $A_{S}$. Then elements in $\mathcal{C}_{E}^{q}$ are sums of the type
\begin{equation}
\label{multi}
a_{K'}^{i}(x, y) \partial_{x^{i}}\wedge \partial_{\xi_{K'}}+
b_{K''}^{j}(x, y) \partial_{y^{j}}\wedge \partial_{\xi_{K''}}+
c_{k,K}(x, y) \xi_k \partial_{\xi_{K}}
\end{equation}
where the sums are over the indices $i, j, k$ and over the multi-indices
$K$ of length $q$, and $K'$ and $K''$ of length $k'=k''=(q-1)$. Moreover, elements in $\mathcal{C}_{S}^{q}$
are those elements with local expressions:
\[ b_{K}^{j}(x) \partial_{y^{j}}\wedge \partial_{\xi_{K}}\]
(note that the coefficient \emph{does not} depend on $y$).

Any automorphism $\phi$ of the vector bundle $A$ (not necessarily covering the identity) induces
a transformation on $\X^{\bullet}(A^*)$ denoted by the same letter. In particular, the family of bundle automorphisms $\phi_{t}^{s}: A\to A$, defined for $t>0$ and $s\in\Gamma(A)$ by:
\[ (x, y, \xi)\longmapsto (x,ty+s(x),\xi),\]
induces transformations $\phi_{t}^{s}:\mathcal{C}_E\to \mathcal{C}_E$. By construction, each $\phi_{t}^{s}$ preserves the bracket on $\mathcal{C}_E$. We also put $\phi_t= \phi_{t}^{0}$. In local coordinates, $\phi_{t}^{s}$ transforms an element (\ref{multi}) into the element:
\begin{multline*}
a_{K'}^{i}(x,ty+s(x)) \partial_{x^{i}}\wedge \partial_{\xi_{K'}}+\\
+\frac{1}{t}\left(b_{K''}^{j}(x,ty+s(x))-a_{K''}^{i}(x,ty+s(x))\frac{\partial s^j}{\partial x^i} \right)\partial_{y^{j}}\wedge \partial_{\xi_{K'}}+\\
+c_{k,K}(x,ty+s(x)) \xi_k \partial_{\xi_{K}}
\end{multline*}
It is clear now that, just as in the Poisson case, we can view $\mathcal{C}_{S}$ inside $\mathcal{C}_{E}$ as the subspace:
\[ \mathcal{C}_{S}:= \{u\in\mathcal{C}_E: \phi_t(u)= \frac{1}{t}u\}.\]
We also introduce the subspace:
\[ \mathcal{C}_{E,\text{lin}}:= \{u \in \mathcal{C}_E: \phi_t(u)= u\}.\]

Similar to $\Omega_{E}$, we have a restriction operation along sections of $E$:
\[ \mathcal{C}_{E}\to \mathcal{C}_{S}, u\mapsto u|_{s}:= \lim_{t\to 0}~t\,\phi_{t}^{s}(u).\]
In local coordinates, if $u$ is an element given by (\ref{multi}), then $u|_{s}$ is the element:
\[
\left(b_{K''}^{j}(x,s(x))-a_{K''}^{i}(x,s(x))\frac{\partial s^j}{\partial x^i} \right)\partial_{y^{j}}\wedge \partial_{\xi_{K'}}
\]

Given an algebroid structure on $A$, viewed as an element $\ell\in \mathcal{C}_{E}^{2}$,
the question of when is the graph $\Graph(s)$ of a section $s\in \Gamma(E)$ a leaf of $\ell$, 
is answered by the following proposition (which is analogous to Proposition \ref{criteria-leaf}):

\begin{prop} 
Given an algebroid structure $\ell$ on $A$, the graph of a section $s\in \Gamma(E)$ is a leaf of $\ell$ if and only if $\ell|_{s}= 0$.
\end{prop}

Next, for $u\in E$ we introduce the linearization $\lins u$ along a section $s$ to be the element:
\[ \lins u:=\lim_{t\to 0}\frac{t\,\phi_{t}^{s}(u)-u|_s}{t}.\]
Hence, for any $u\in \mathcal{C}_{E}$ we have the formal expansion
\[
\phi_{t}^{s}(u)\sim \frac{1}{t} u|_{s}+ \lins u + o(t).
\]
Just as in the Poisson case, for $\ell\in \mathcal{C}_{E}^{2}$, we define the operator
\[ \d_{\ell,s}:= [\lins \ell, -]: \mathcal{C}_{S}\rmap \mathcal{C}_{S}\]
which rises the degree by $1$. The operator $\d_{\ell,s}$ is defined on the entire $\mathcal{C}_{E}$. The fact that it restricts to $\mathcal{C}_{S}$ is a formal consequence of the expansion above and the definition of $\mathcal{C}_{S}$. It is easy to see that when $\ell=\ell_0$ and $s=0$, the resulting complex $(\mathcal{C}^\bullet_{S},\d_{\ell_0,0})$ coincides with the complex $(\Omega^{\bullet-1}(A_S; E),\d)$ computing the Lie algebroid cohomology of $\ell_0$ with coefficients in $\nu_S=E$, which appears in the statement of the theorem. 

The analogue of Proposition \ref{linearization} follows by linearizing the structure equation $[\ell,\ell]= 0$ along a section $s$. To do so, we first apply $\phi_{t}^{s}$ and then formally expand in powers of $t$. We obtain:

\begin{prop} 
For any algebroid structure $\ell$ on $A$ and any $s\in \Gamma(E)$,
\[ \d_{\ell,s}(\ell|_s)=0.\]
\end{prop}

All the analytical aspects discussed in Section \ref{sec:sobolev} go through by the same type of arguments, with same care to obtain the bound $\kappa>\frac{n}{2}$ (in particular, replace \cite[Corollary 9.7]{Palais} by the stronger \cite[Theorem 9.6]{Palais}).

The proof now follows the same pattern as before. For any algebroid structure $\ell$ we want to look for sections $s$ such that the element:
\[ c(\ell,s)=\ell|_s,\]
vanishes. Hence, one considers the functional $\Phi_{\ell}:\mathcal{C}^{1,(\kappa)}_{S}\to\Rr$ defined by:
\[ \Phi_{\ell}(s):= ||c(\ell,s)||_{(\kappa-1)}^{2},\]
and whose whose differential is
\[ (\d\Phi_{\ell})|_s(w)= -\langle \d_{\ell,s}(w),c(\ell,s)\rangle _{(\kappa-1)}.\]
For the original Lie algebroid structure $\ell_0$, the zero section $s=0$ is a critical point of this functional, and the Hessian at this point is:
\[ (\d^2\Phi_{\ell_0})|_0(w_1,w_2)= \langle \d_{\ell_0,0}(w_1),\d_{\ell_0,0}(w_1)\rangle _{(\kappa-1)}.\]
Hence, if denote by $K$ the kernel of
\[ \d_{\ell_0, 0}: \mathcal{C}^{1,(\kappa)}_{S}\to \mathcal{C}^{2,(\kappa-1)}_{S}\]
and by $W$ its orthogonal complement, then the restricted functional
\[ \Phi_{\ell_0}|_{W}:W\to\Rr\] 
will have a strongly non-degenerate critical point at the origin $0\in W$ (here we use the fact that the complex $\Omega^{\bullet}(A_S;E)$ is elliptic).

The conclusion is that for $s\in K$ close enough to zero and $\ell$ close enough to $\ell_0$, we find that there exists an element
\[ \tilde{s}_{\ell,s}= s+ w_{\ell, s}\in s+ W\]
which is a critical point of $\Phi_{l}|_{s+ W}$. Exactly as in the Poisson case, one deduces from this that the element $c(\ell,\tilde{s}_{\ell,s})$ lies in the kernel of the operator:
\[ \lap_{\ell,\tilde{s}}:= \d_{\ell_0,0}\d_{\ell,\tilde{s}}^{*}+\d_{\ell_0,0}^*\d_{\ell,\tilde{s}}.\]
When $\ell=\ell_0$ and $s=0$ the ``Laplacian'' $\lap_{\ell,0}$ is an isomorphism, so we conclude that, if $\ell$ is close enough to $\ell_0$ and $s\in K$ is close enough to $0$, then $\lap_{\ell,\tilde{s}}$ is also an isomorphism. Therefore, we must have $c(\ell,\tilde{s}_{\ell,s})=0$, so that $\tilde{s}$ defines a leaf of $\ell$. To complete the proof, one remarks that since the complex $\mathcal{C}^{\bullet}_{S}=\Omega^{\bullet-1}(A_S;E)$ is elliptic and $\mathcal{C}^{0}_{S}=0$ one has 
\begin{align*}
K&=\ker (\d_{\ell_0, 0}:\mathcal{C}^{1,(\kappa)}_{S}\to \mathcal{C}^{2,(\kappa-1)}_{S})\\
 &=\ker (\d_{\ell_0, 0}:\mathcal{C}^{1}_{S}\to \mathcal{C}^{2}_{S})=H^0(A_S;E).
\end{align*}

\bibliographystyle{amsplain}

\begin{thebibliography}{11}


\bibitem{AS} M.~Atiyah and I.~M.~Singer, The index of elliptic operators. V,
\emph{Annals of Math.}~\textbf{93} (1971), 139--149.


\bibitem{Conn} J.~Conn, Normal forms for smooth Poisson structures,
\emph{Annals of Math.}~\textbf{121} (1985), 565--593.

\bibitem{Cr} M.~Crainic, Differentiable and algebroid cohomology, van
Est isomorphisms, and characteristic classes,
\emph{Comment.~Math.~Helv.~}\textbf{78} (2003), no. 4, 681--721.

\bibitem{CrFe3} M.~Crainic and R.~L.~Fernandes, Integrability of Poisson
  brackets, \emph{J.~Differential Geom.}~\textbf{66} (2004), 71--137.

\bibitem{CrFe2} M.~Crainic and R.~L.~Fernandes, Rigidity and flexibility in Poisson geometry,
\emph{Travaux math\'ematiques} \textbf{XVI} (2005), 53--68.

\bibitem{CrFe1} M.~Crainic and R.~L.~Fernandes, Integrability of Lie
brackets, \emph{Ann.~of Math.~(2)} \textbf{157} (2003), 575--620.

\bibitem{CrMo} M.~Crainic and I.~Moerdijk, Deformations of Lie
brackets: cohomological aspects, \emph{J.~Eur. Math.~Soc.~}\textbf{10} (2008), 1037--1059.

\bibitem{DuWa} J.P.~Dufour and A.~Wade, Stability of higher order singular points of Poisson manifolds and Lie algebroids, \emph{Ann.~Inst.~Fourier (Grenoble)} \textbf{56} (2006), no. 3, 545--559.

\bibitem{Gilk} P.~Gilkey, \emph{Invariance theory, the heat equation, and the Atiyah-Singer index theorem}.
Second edition. Studies in Advanced Mathematics. CRC Press, Boca Raton, FL, 1995.

\bibitem{GiLu} V.L.~Ginzburg and J.-H.~Lu, Poisson cohomology of Morita-equivalent Poisson manifolds,
\emph{Internat.~Math.~Res.~Notices} (1992), no. 10, 199--205.

\bibitem{Hr} M.~Hirsch, Stability of stationary points and cohomology of
groups, \emph{Proc.~Amer. Math.~Soc.~}\textbf{79} (1980), 191--196.

\bibitem{Hr2} M.~Hirsch, \emph{Differential Topology}, Graduate Texts in Mathematics 33, Springer-Verlag, New-York, 1976.

\bibitem{Ko} K.~Kodaira, A theorem of completeness of characteristic systems for analytic families of compact submanifolds of complex manifolds,
\emph{Ann.~of Math.~(2)} \textbf{75} (1962), 146--162. 

\bibitem{LaRo} R.~Langevin and H.~Rosenberg, On stability of compact
leaves and fibrations, \emph{Topology} \textbf{16} (1977),
107--111.

\bibitem{Nij} A.~Nijenhuis, Strong derivatives and inverse mappings,
\emph{Amer.~Math.~Monthly} \textbf{81} (1974), 969--980.

\bibitem{Palais} R.~S.~Palais, \emph{Foundations of global non-linear analysis},
W. A. Benjamin, Inc., New York-Amsterdam 1968.

\bibitem{Rb} G.~Reeb, Sur certaines propri\'et\'es topologiques des
vari\'et\'es feuillet\'ees,
\emph{Publ.~Inst. Math. Univ.~Strasbourg} \textbf{11} (1952), 5--89
and 155--156.

\bibitem{Shubin} M.~A.~Shubin, \emph{Pseudodifferential operators and spectral theory},
Springer Series in Soviet Mathematics. Springer-Verlag, Berlin, 1987.

\bibitem{St} D.~Stowe, Stable orbits of differentiable group actions,
\emph{Trans.~Amer.~Math.~Soc.~}\textbf{277} (1983), 665--684.

\bibitem{Th} W.~Thurston, A generalization of the Reeb stability
theorem, \emph{Topology} \textbf{13} (1974), 347--352.

\bibitem{Vor} Y.~Vorobjev, Coupling tensors and Poisson geometry near
a single symplectic leaf, \emph{in} Lie algebroids and related
topics in differential geometry (Warsaw, 2000), 249--274,
\emph{Banach Center Publ.}~\textbf{54}, Polish Acad.~Sci., Warsaw, 2001.

\bibitem{Wein1} A.~Weinstein, The local structure of Poisson manifolds,
\emph{J.~Differential ~Geometry}~\textbf{18} (1983), 523--557.

\end{thebibliography}
\def\lllll{}
 
\end{document}